\documentclass[11pt]{article}
\usepackage{enumerate}
\usepackage{amssymb}
\usepackage{amsthm}
\usepackage{comment}
\usepackage{amsmath}
\usepackage{graphicx}
\usepackage{setspace}
\usepackage[margin=1in]{geometry}
\usepackage{subfigure} 
\usepackage{float}  
\usepackage{color} 
\DeclareMathOperator{\Win}{Win}
\DeclareMathOperator{\best}{BEST}

\DeclareMathOperator{\ceil}{ceil}

\title{A Secretary Problem with a Sliding Window for Recalling Candidates}
\author{Shan-Yuan Ho and Abijith Krishnan}
\date{\today}
\newtheorem{theorem}{Theorem}[section]

\newtheorem{lemma}[theorem]{Lemma}

\begin{document}
\begin{titlepage}
\begin{singlespace}
\maketitle
\end{singlespace}

\begin{abstract}
\begin{singlespace}
The Sliding Window Secretary Problem allows a window of choices to the Classical Secretary Problem, in which there is the option to choose the previous $K$ choices immediately prior to the current choice. We consider a case of this sequential choice problem in which the interviewer has a finite, known number of choices and can only discern the relative ranks of choices, and in which every permutation of ranks is equally likely. We examine three cases of the problem: (i) the interviewer has one choice to choose the best applicant; (ii) the interviewer has one choice to choose one of the top two applicants; and (iii) the interviewer has two choices to choose the best applicant. The form of the optimal strategy is shown, the probability of winning as a function of the window size is derived, and the limiting behavior is discussed for all three cases. 
\end{singlespace}
\end{abstract}
\end{titlepage}
\newpage

\section{Introduction} \label{sec:Intro}
The classical secretary problem is a well-known decision theory problem, and the solution to the problem was first proven by Lindley (1961) and Dynkin (1963). Ferguson presents the problem as follows \cite{Ref1}: an interviewer sees a sequence of $N$ applicants one at a time, and must decide whether to accept or to reject an applicant immediately after seeing the applicant. The interviewer's decision is solely based on the relative ranks of previous applicants. No rejected applicant can be recalled, and the interviewer must make exactly one choice. Success occurs if the top applicant is chosen. For large $N$, the optimal strategy for the problem is the following threshold rule: reject a threshold number of applicants $\sim \frac{N}{e}$, and choose the next best applicant to appear. The interviewer wins with a probability that is approximately $\frac{1}{e} \approx 0.37$ with this strategy.

The classical secretary problem has many applications. For example, the classical secretary problem has been applied to the behavior of a person searching for the best gas station or best restaurant after agreeing to look through a fixed number of stores. In fact, Seale and Rapoport \cite{Ref15} found that when presented with a scenario equivalent to the secretary problem, a majority of the fifty people in the study used a threshold rule, with the deviation of their threshold from the optimal threshold accounted for by an additional cost for the time spent before making a decision. The classical secretary problem can also be applied to data stream mining, in which a sampler collects and analyzes data real-time from sensors, computer programs, or web traffic. For example, Girdhar and Dudek \cite{Ref12} used a version of the secretary problem to model the optimal strategy for a robot probing a landscape to find the best location to place a sensor by taking a large number of pictures and assigning a score to each picture based on the variety of colors. In addition, Das \cite{Ref16} experimentally tested an algorithm that used the optimal strategy from a secretary problem to collect plankton that best represented a species responsible for toxic algal blooms.

However, the classical secretary problem does not perfectly apply to the above situations. Realistically, the interviewer would have more time to decide on an applicant. Similarly, a person deciding while driving whether to stop at a particular gas station or restaurant would have some ability to backtrack and choose a previous store. As a result, Seale and Rapoport's findings \cite{Ref15} could be extended to more realistic scenarios if the decision-maker was given more time to make a decision. In addition, providing the decision-maker with more time would be beneficial for data stream mining. Ajtai, Megiddo, and Waarts \cite{Ref10} note that the classical secretary problem could be applied to choosing records of highest interest from a large data set or choosing images from a large digital library, but also that allowing for limited backtracking would make the application more realistic. We thus consider a secretary problem proposed in 2009 by Beccheti and Koutsoupias \cite{Ref9} in which the interviewer can keep the last $K$ applicants as possible choices and hence has a sliding window of size $K$.

In this paper, we study two cases of the Sliding-Window Secretary Problem with a fixed window size of $K$: a payoff of 1 for choosing the best applicant and 0 otherwise, \textit{the Best-1 case}; and a payoff of 1 for choosing one of the top 2 applicants and 0 otherwise, \textit{the Best-2 case}. We additionally study the \textit{2-Choice case}, in which the interviewer can choose two applicants and wins only if either of them are the best applicant overall. We discuss previous variations of the secretary problem in Section \ref{Ap-1}. Then, for each case of the Sliding-Window Secretary Problem, we outline the effect of changes in the window size on the probability of winning, analyze special cases for the window size, provide a recursive solution that computes the probability of winning, and finally analyze limiting cases of the recursive solutions. We discuss \textit{the Best-1 case} in Section \ref{Best1Case}, \textit{the Best-2 case} in Section \ref{Best2}, \textit{the 2-Choice case} in Section \ref{2Choice}, and concluding remarks and future directions in Section \ref{concluding}. 

\section{Background and History}  \label{Ap-1}
The classical secretary problem's solution was first proven by Lindley in 1961 \cite{Ref6} and Dynkin in 1963 \cite{Ref7}. Their results are discussed in Section \ref{sec:Intro}. Many variations of the secretary problem have been studied in the past 60 years. We highlight a few variations, but not all. \\

\noindent
\textit{Finding the Best Applicant with Multiple Choices}: Gilbert and Mosteller \cite{Ref2} offer a variation of the secretary problem in which an interviewer can choose $r$ people from a pool of $N$ applicants and wins if one of the $r$ people is the best applicant. For the case $r=2$, they show that for large $N$ there are two optimal thresholds for each choice, $\frac{N}{e}$ and $\frac{N}{e^{1.5}}$, and an optimal probability of winning of $e^{-1}+e^{-1.5}$.  They extend their analysis to general $r$ and find the asymptotic behavior of the problem. Their results are numerically derived and cannot be explained analytically. \\

\noindent
\textit{Choosing the Best or the Second Best Applicant}:
Gilbert and Mosteller also analyze a secretary problem in which the interviewer wins if the best or second best applicant is chosen. There are two threshold values, $d_{1}^{*}$ and $d_{2}^{*}$ in the optimal strategy. The interviewer passes $d_{1}^{*}$ applicants, then chooses the next applicant better than all previous applicants. If an applicant has not been chosen by index $d_{2}$, the interviewer now chooses the best or second best applicant out of all previous applicants. They find $d_{1}^{*} \approx 0.347N$ and $d_{2}^{*} = \frac{2N}{3}$ and that the optimal probability of success is approximately $0.574$ for $N$ large. This problem is discussed in more detail in the context of the Sliding Window Secretary Problem in Section \ref{Best2}. \\

\noindent
\textit{The Ability to Recall a Candidate with a Fixed Probability}:
Another variation allows the interviewer to recall a previous applicant with a fixed probability, as seen in Petrucelli \cite{Ref8}. The applicant currently being interviewed can accept the job with a probability of $q \leq 1$, and if the interviewer decides to choose some previous applicant, the probability of the previous applicant accepting the job is $p \leq q$.  While the probability of winning increases as $p$ increases, the probability of winning with a nonzero value of $p$ approaches the probability of winning with $p=0$ for $N$ large. \\

\noindent
\textit{The Best Expected Rank}:
The payoff for the secretary problem is now the value of the rank of the applicant, and the interviewer seeks to minimize the expected rank. Chow, Moriguti, Robbins, and Samuels \cite{Ref3} show that as $N$ approaches infinity, the best expected rank approaches $3.87$. \\

\noindent
\textit{Maximizing the Expected Rank with the Ability to Choose More than One Applicant}:
Ajtai, Megiddo, and Waarts \cite{Ref10} extend the work of Chow \textit{et al.} by looking at the best expected rank, given $r$ choices. They devised algorithms for this process and found that the best expected sum of the $z^{\mbox{\tiny{th}}}$ power of the ranks of the $r$ choices is between $\frac{r^{z+1}}{z+1}+O(k^{z})$ and $\frac{r^{z+1}}{z+1}+C(z) r^{z+0.5} \log r$, where $C(z)$ is a value that depends on $z$. \\

\noindent
\textit{Recalling Previous Candidates}:
Using the same payoff as Chow \textit{et al.}, Goldys considered the problem in which an interviewer tries to achieve the best expected rank with a sliding window of size $2$. He showed that as $N$ approaches $\infty $, the best expected rank approaches approximately $2.57$ \cite{Ref5}.

\section{The Sliding-Window Problem: The \emph{Best-1 Case}} \label{Best1Case}
In this section we study the Secretary Problem with a Sliding Window of choices. The interviewer knows the number of applicants $N$ and can choose any of the last $K$ applicants, for some fixed $K$. Let the \textit{index} of an applicant be its position in a sequence of applicants. Then, we define the window to be the set of $K$ consecutive applicants that the interviewer can choose from, such that the smallest index in the window contains the applicant who must be rejected or accepted before a new applicant can be interviewed. Each applicant has a distinct rank and is seen sequentially in a randomized order. Let $R(m)$ be a bijective function from $[1,N]$ to $[1,N]$ that returns the absolute \textit{rank} of the applicant at index $m$, with $1$ representing the best rank. However, the interviewer can only rank the applicants seen so far and thus can only discern relative ranks. We seek the optimal strategy for finding the best applicant, in which the payoff is 1 for choosing the best and 0 otherwise. 

When $K=1$, the problem is identical to the Classical Secretary Problem. While Lindley (1961) and Dynkin (1963) have proven the secretary problem earlier, we refer to the 1966 paper of Gilbert and Mosteller \cite{Ref2}. Gilbert and Mosteller derived the optimal strategy and the optimal thresholds. For large $N$, it is optimal to pass over approximately $\frac{N}{e}$ applicants and then choose the next best applicant. This gives $\Pr(\Win)\approx \frac{1}{e}$. We present their proofs in \ref{Ap0} for intuition for later proofs.

Let a \textit{candidate} be an applicant that provides a strictly nonzero probability of winning from the perspective of the interviewer if chosen. Specifically in the \textit{Best-1 case}, a candidate is located in the current window and has the best rank out of all seen applicants. Because rejecting applicants who are not candidates does not reduce the probability of winning, we adopt a \textit{sliding rule} in which applicants are interviewed to advance the window until a candidate is at the smallest index of the window. We now extend the optimal strategy of the Classical Secretary Problem and show that in order to maximize the probability of winning the interviewer must reject a particular number of applicants and then accept the first candidate to appear, due to the following concept from Gilbert and Mosteller (1966) \cite{Ref2}:
we choose candidate $i$ in our window if and only if 
\begin{equation} \label{InEq}
\Pr(\Win \mid \mbox{Choosing Candidate } i) > \Pr(\Win \mid \mbox{Rejecting Candidate } i),
\end{equation} 
because the interviewer only chooses an applicant that provides a higher probability of winning if chosen than if rejected. 

\begin{theorem}\label{Cor1}
The optimal strategy for the Best-1 case of the Sliding-Window Secretary Problem is to reject the first $d^*$ applicants for some integer $d^*\geq 0$, and then to choose the next candidate with the sliding rule.
\end{theorem}

\begin{proof}
Let $S=\{i \in [1,N] \mid \mbox{ Inequality \eqref{InEq} holds}\}$. Because the interviewer has seen $i+K-1$ applicants when the window starts at $i$, $\Pr(\Win \mid \mbox{Choosing Candidate } i) = \frac{i+K-1}{N}$. Because the probability that the best applicant lies between $i+K$ and $N$ decreases as $i$ increases, $\Pr(\Win \mid \mbox{Rejecting Candidate } i)$ decreases in $i$. If applicant $N-K$ is a candidate, because the last applicant is the best with probability $\frac{1}{N}$, the probability of winning and rejecting candidate $N-K$ is $\frac{1}{N}$. A sketch of $\Pr(\Win \mid \mbox{Choosing Candidate } i)$ and $\Pr(\Win \mid \mbox{Rejecting Candidate } i)$ is shown in Figure \ref{T21} to provide intuition for the remaining part of the proof. Thus Inequality \eqref{InEq} holds for $i=N-K$. Because $\Pr(\Win \mid \mbox{Choosing Candidate } i)$ strictly increases and $\Pr(\Win \mid \mbox{Rejecting Candidate } i)$ decreases in $i$, all elements in $S$ are consecutive integers. Because $S$ is nonempty, there is a least element in $S$, which we call $d^{*}+1$. Thus there is a $d^*$ such that the first $d^*$ applicants should be rejected, and the first candidate after $d^*$ should be accepted. 
\end{proof}

\begin{figure} [h]
\centering
\includegraphics[width=\textwidth]{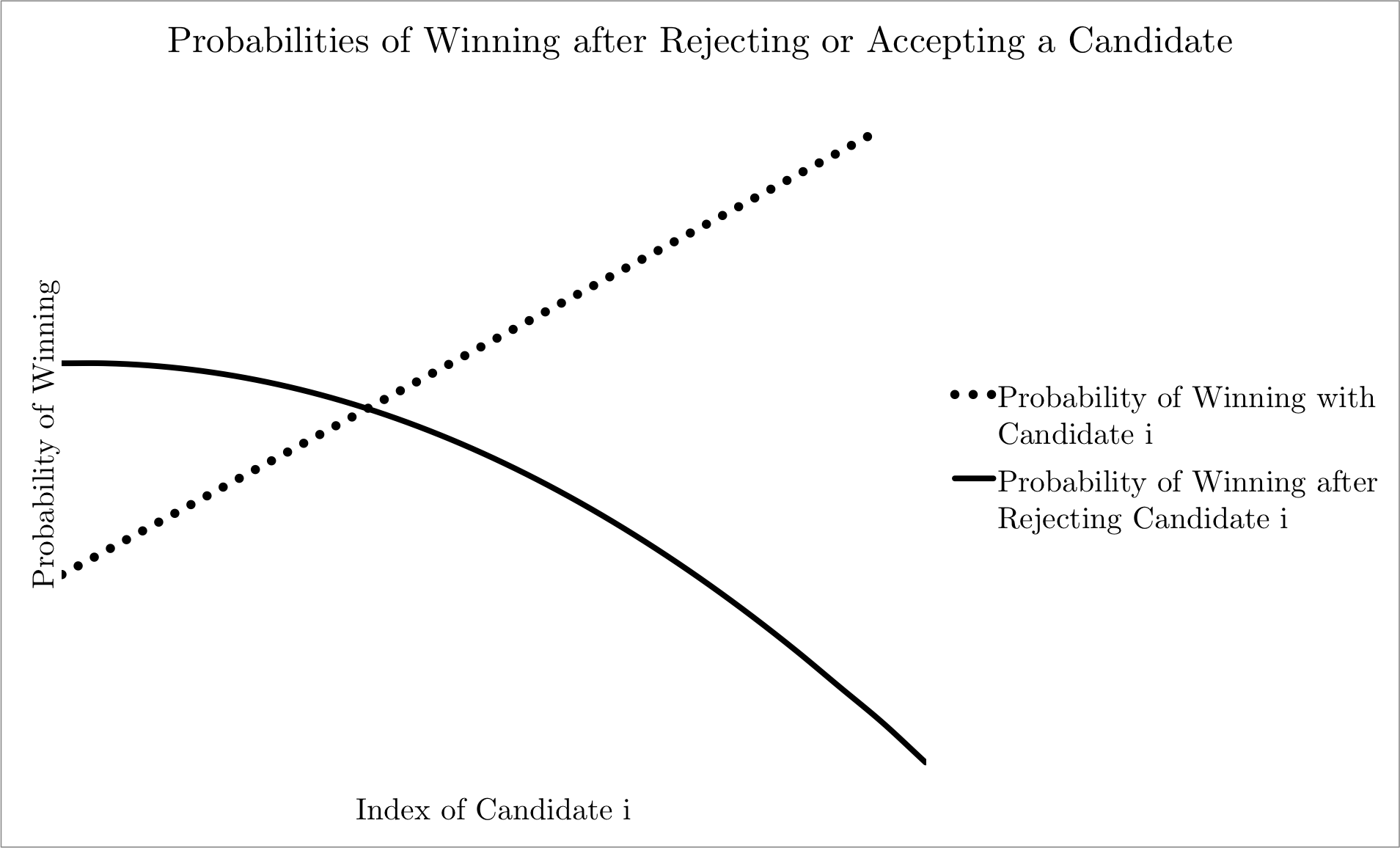}
\caption{A Pictorial Representation of Theorem \ref{Cor1}. The probability of winning with candidate $i$ is strictly increasing, while the probability of winning after rejecting candidate $i$ is decreasing. The threshold value occurs where the two lines meet.}
\label{T21}
\end{figure}

By the definition of $d^*$ in Theorem \ref{Cor1}, the \textit{optimal strategy for the Best-1 case} is to reject the first $d^*$ applicants and use the sliding rule to accept the next candidate. Note that even though the first $d^*$ applicants are skipped, their relative ranks are still used to determine if an applicant is a candidate.

\subsection{Special Cases for K} \label{casework}

We first characterize the probability of winning for $K=2$ and $K=N-1$ in this section. 

For a large window size of $K=N-1$, the only possible 
threshold values are $0$ or $1$. 
Suppose $d=1$. Failure occurs only when the best applicant is skipped, i.e., $R(1)=1$. This event occurs with probability $\frac{1}{N}$. Now suppose $d=0$. Failure only occurs if the second best is at index 1 and the best is at index $N$, i.e., $R(1)=2$ and $R(N)=1$. This event occurs with probability $(\frac{1}{N})(\frac{1}{N-1})$. Thus, $d^{*}=0$.

Now consider a small window size of $K=2$. Let $j$ be the index of the best candidate, and $d$ be an arbitrary threshold. If $j\leq d$, the interviewer loses, and if  $j=d+1$ or $j=d+2$, the interviewer wins. If $j>d+2$, the interviewer wins if there are no candidates before $j$. If there is an applicant $i>d$ better than all previous applicants, then applicant $i+1$ must be better than applicant $i$ so that applicant $i$ is not a candidate. Thus the sequence of applicants between $i$ and $j-1$ must form a sequence of strictly improving ranks. The probability that $R(x)$ is better than the rank of all preceding applicants is $\frac{1}{x}$, and the probability that $i$ is the first applicant better than the first $d$ applicants is $\frac{d}{i-1}$. Thus, if $i$ is the first applicant better than the first $d$ applicants, the probability that there are no candidates before $j$ is $\frac{d}{i-1} \prod _{x=i}^{j-1} \frac{1}{x} = \frac{d(i-2)!}{(j-1)!}$. Since each value of $j$ occurs with probability $\frac{1}{N}$, and $j=d+1$ and $j=d+2$ guarantee wins, when we sum the probabilities for all values of $i$ and $j$, we find
\begin{equation} \label{Eq7}
\Pr(\Win \mid d) = \frac{2}{N} + \frac{d}{N}  \sum_{j=d+3}^{N} \frac{\sum_{i=d+1}^{j} (i-2)!}{(j-1)!}.
\end{equation}

Values of the summation in Equation \eqref{Eq7} for $N=100$ and various values of $d$ are in \ref{Ap2}.

Analyzing the results of a simulation for small values of $N$ and $K$ in \ref{Ap1} suggests that as $K$ increases for fixed $N$, $\Pr(\Win)$ increases and $d^*$ decreases. Therefore, we first prove that $\Pr(\Win)$ strictly increases as $K$ increases for fixed $N$.

\begin{lemma}  \label{Lemma2}
Let $d^{*}_{K}$ and $d^{*}_{\kappa}$  be the optimal thresholds for windows $K$ and $\kappa$ respectively. If $\kappa<K$, then $\Pr(\Win \mid \kappa, d^{*}_{\kappa}) < \Pr(\Win \mid K, d^{*}_{K})$. 
\end{lemma}

\begin{proof}
We show that $\Pr(\Win \mid \kappa, d^{*}_{\kappa}) < \Pr(\Win \mid K, d^{*}_{\kappa}) \leq \Pr(\Win \mid K, d^{*}_{K})$. Because $d^*_K$ is optimal for a window size of $K$, $\Pr(\Win \mid K, d^{*}_{\kappa}) \leq \Pr(\Win \mid K, d^{*}_{K})$. We now prove that $\Pr(\Win \mid \kappa, d^{*}_{\kappa}) < \Pr(\Win \mid K, d^{*}_{\kappa})$. We define $j$ to be the index of the best applicant.

A window of $K$ provides the interviewer with at least the same winning sequences as a window of $\kappa$, for the same $d^*_{\kappa}$, because the interviewer can ignore the last $K-\kappa$ applicants in the window. In addition, there exists a sequence in which a candidate appears before $j-\kappa+1$, but after $j-K$. Therefore, with this sequence, the interviewer loses with a window of $\kappa$ but wins with a window of $K$. Thus $\Pr(\Win \mid \kappa, d^{*}_{\kappa}) < \Pr(\Win \mid K, d^{*}_{\kappa})$.
\end{proof}
Now we prove that $d^*_K$ decreases as $K$ increases for fixed $N$. 
\begin{lemma}  \label{Lemma3}
Let $d^{*}_{K}$ and $d^{*}_{\kappa}$  be the optimal passing thresholds for windows $K$ and $\kappa$, respectively, and $j$ be the index of the best applicant. If $\kappa<K$, then $d^{*}_{K} \leq d^{*}_{\kappa}$. 
\end{lemma}

\begin{proof} From the proof of Theorem \ref{Cor1}, if applicant $i$ is a candidate, $\Pr(\Win \mid \mbox{Choosing }i \mbox{, } K)>\Pr(\Win \mid \mbox{Choosing }i \mbox{, } \kappa)$. Because $j \in [(i+\kappa), N]$ occurs with higher probability than $j \in [(i+K), N]$,  $\Pr(\Win \mid \mbox{Rejecting }i \mbox{, } K)\leq \Pr(\Win \mid \mbox{Rejecting }i \mbox{, } \kappa)$. By Theorem \ref{Cor1}, the smallest integer $i$ such that Inequality \eqref{InEq} holds is $(d^*_{\kappa}+1)$. It follows from the previous inequalities that Inequality \eqref{InEq} holds for a window size of $K$ at index $(d^*_{\kappa}+1)$. Therefore, because $d^{*}_{K}+1$ is the least index for which Inequality \eqref{InEq} holds if applicant $d^{*}_{K}+1$ is a candidate, $d^{*}_{K} \leq d^{*}_{\kappa}$. 
\end{proof}

Finally we present the exact and asymptotic solutions to the secretary problem for a window size of $K\geq \frac{N}{2}$.

\begin {comment}
\begin{theorem}  \label{Thm2}
For $n$ even, if $k=\frac{n}{2}$, then the optimal $d^{*}=0$ or the optimal $d^{*}=1$. 
\end{theorem}
\begin{proof}
Let $k=\frac{n}{2}$. There are two cases to investigate: $R(1) = \best(\{R(x) \mid x \in [1,k]\})$ and $R(1) \neq \best(\{R(x) \mid x \in [1,k]\})$. 

If $R(1) = \best(\{R(x) \mid x \in [1,k]\})$, then the probability of winning given that you skip Applicant 1 is $\frac{1}{2}$, because there is a $\frac{1}{2}$ chance that Applicant 1 is the best applicant. So skipping Applicant 1 and not skipping Applicant 1 would give the same chance of success. 

If $R(1) \neq \best(\{R(x) \mid x \in [1,k]\})$, then the interviewer the window until the window is at position $a$ such that $a=n-k+1$ or $R(a) = \best(\{R(x) \mid x \in [a,a+k-1]\})$. The probability that $a$ is the best applicant is $\frac{k+a-1}{n}>\frac{1}{2}$, so the interviewer will not skip Applicant $a$. Therefore, the optimal $d^{*} \ngtr 1$. 

Therefore, the optimal $d^{*} = 0 \mbox{ or } 1$.
\end{proof}
\end{comment}
\begin{theorem}  \label{ThmM}
Let $d^*$ be the optimal threshold number of applicants to reject.
\begin{enumerate} [(i)]
\item If $K > \frac{N}{2}$, then $d^{*}=0$.
\item If $K = \frac{N}{2}$, then $d^{*}=0$ or $d^*=1$.
\item For $N>>1$ and $K \geq \frac{N}{2}$, $\Pr(\Win) \approx 2-\frac{K}{N} + \ln \frac{K}{N}$. 
\end{enumerate}
\end{theorem}

\begin{proof}
(i) and (ii): Applicant 1 may or may not be a candidate. First let Applicant 1 be a candidate. If applicant 1 is chosen, $\Pr(\Win)=\frac{K}{N}$. If applicant 1 is rejected, no candidates are between $2$ and $K$, and so the window slides past $K$, and the remaining $N-K$ applicants are seen. The best applicant among the $N-K$ applicants is the best overall with probability $\frac{N-K}{N}$. Thus for $K>\frac{N}{2}$, applicant 1 should be accepted, and for $K=\frac{N}{2}$, accepting or rejecting applicant 1 provides equal probabilities of winning, $\frac{1}{2}$. Now let applicant 1 not be a candidate. By the sliding rule, the window starts at some index $i>1$ where $i$ is a candidate. The problem reduces to a new problem with $N-i+1$ applicants, window size of $K$, and a candidate at index $1$. Because $K>\frac{N-i+1}{2}$, the first candidate should be selected. Therefore, $d^*=0$ for $K>\frac{N}{2}$, and $d^*$ is either 0 or 1 for $K=\frac{N}{2}$.

(iii): The best applicant at index $j$ is guaranteed to be chosen if $j \in [1,K]$, which occurs with probability $\frac{K}{N}$. If index $j \in [K+1, N]$, then $j$ will be chosen if no candidates are before $j-K+1$. For $m<j-K+1$ the probability that $m$ is a candidate is $\frac{1}{m+K-1}$. By the sliding rule, there cannot be more than one candidate in $[1,K]$. Therefore, we sum $\frac{1}{m+K-1}$ over all $m \in [1,j-K+1]$, and take the complement of the sum to find the probability of not stopping before $j$. Each value of $j$ occurs with probability $\frac{1}{N}$. By the Total Probability Theorem, we add up the probability of winning for all possible values of $j$ and find
\begin{equation} \label{Eq11}
\Pr \left(\Win \mid K\geq \frac{N}{2}\right) = \frac{K}{N} + \frac{1}{N} \sum _{j=K+1}^{N} \left( 1-\sum_{m=1}^{j-K} \frac{1}{m+K-1} \right).
\end{equation}

For large $N$ we approximate the sums in Equation \eqref{Eq11} as integrals. The worst approximation of the inner sum occurs when the integral approximates only one term in the summation: $\frac{1}{K}$. Because the function in the integral has initial value $\frac{1}{K}$, final value $\frac{1}{K+1}$ and is strictly decreasing, the integral approximates the sum with an error on the order of $\frac{1}{K^2}$.  Therefore, because $N$ and $K$ are large, the integral approximation is acceptable, and is similarly acceptable for the outer sum. If we let $x=\frac{K}{N}$, $y=\frac{m}{N}$ and $z=\frac{j}{N}$, 
\begin{equation*}
\Pr(\Win \mid K\geq \frac{N}{2}) \approx x +  \int \limits_{x}^{1} \left( 1-\int\limits_{0}^{z-x} \frac{\mathrm{d}y}{x+y} \right) dz  = 2-\frac{K}{N} + \ln \frac{K}{N}. \qedhere
\end{equation*}
\end{proof}

\subsection{A Recursive Formula for the Probability of Winning}
We now analyze the problem for some window size $K$ and some threshold value $d$ of automatically rejected applicants. We divide the sequence of applicants after $d$ into blocks of $K$ because the sliding rule guarantees that no block of $K$ has more than one candidate. Let $f_{q}(a)$ be the probability of stopping between $(d+(q-1)K+1)$ and $a$, where $q=\lceil\frac{a-d}{K}\rceil$. Because no applicant before $d$ is chosen, $f_{q}(a)=0$ for $q<1$ and $a\leq d$. We present a recursive formula for $f_{q}$.

\begin{lemma}  \label{importantthm}
For $q>0$,
\begin{equation} \label{Eq01} 
f_{q}(a)=\sum_{m=d+(q-1) K + 1} ^{a} \frac{1}{m+K-1}  \left(1 - \sum_{r=-1} ^{q-2} f_{r}(d+r K) - f_{q-1}(m-K) \right).
\end{equation}
\end{lemma}

\begin{proof} 
By the sliding rule, the window stops sliding when a candidate is at the smallest index of the window. The probability that an applicant at some index $m$ is a candidate is $\frac{1}{m+K-1}$. However, $m$ will not be reached if a candidate is between $(d+1)$ and $(m-K)$, so we subtract the probability that a candidate appears in the previous $q-2$ blocks or between indices $(d+(q-2)K+1)$ and $m-K$. Summing these probabilities for all values of $m$ in $[(d+(q-1)K+1),a]$ yields
\begin{equation*} 
f_{q}(a)=\sum_{m=d+(q-1) K + 1} ^{a} \frac{1}{m+K-1}  \left(1 - \sum_{r=-1} ^{q-2} f_{r}(d+r K) - f_{q-1}(m-K) \right). \qedhere
\end{equation*}
\end{proof}

We now find $\Pr(\Win)$ for a particular $N$, $K$, and threshold index $d$. Let $\sigma_{q}(a)$ be the probability of winning with a candidate in $[1,a]$, where $q= \lceil \frac{a-d}{K}\rceil$. Because no applicants before $d$ are chosen, $\sigma_{q}(a)=0$ for $q<1$ and $a\leq d$. We present a recursive formula for $\sigma_{q}$.

\begin{theorem} \label{impttheorem}
For $q>0$, 
\begin{equation}\label{Eq21}
\sigma_{q}(a)=\sigma_{q-1}(d+(q-1) K)+\frac{1}{N} \sum_{j=d+(q-1) K+1} ^{a} \left(1-\sum_{r=-1}^{q-2}  f_{r}(d+r K) - f_{q-1}(j-K)\right).
\end{equation}
\end{theorem}

\begin{proof}
The probability of winning with a candidate between $1$ and $a$ is the sum of the probability of winning with a candidate between indices $1$ and $(d+(q-1)K)$ and the probability of winning with a candidate between indices $(d+(q-1)K+1)$ and $a$. The probability of the former event is $\sigma_{q-1}(d+(q-1)K)$, so we find the probability of the latter event. The probability that the best applicant is at an index $j$ is $\frac{1}{N}$, and the probability of stopping before $j$ is subtracted as in the proof of Lemma \ref{importantthm}. Summing the probability of winning at index $j \in [(d+(q-1)K+1), a]$ yields
\begin{equation*}
\sigma_{q}(a)=\sigma_{q-1}(d+(q-1) K)+\frac{1}{N} \sum_{j=d+(q-1) K+1} ^{a} \left(1-\sum_{r=-1}^{q-2}  f_{r}(d+r K) - f_{q-1}(j-K)\right). \qedhere
\end{equation*}
\end{proof}
If we let $q'=\lceil \frac{N-d}{K} \rceil$ we see that $\Pr(\Win)=\sigma_{q'}(N)$.

Analyzing the Sliding-Window Secretary Problem for large $N$ provides intuition for the optimal strategy for any $N$. In the classical secretary problem, the $\Pr(\Win)$ depends on $\frac{d^*}{N}$ for $N$ large. We now show that $\Pr(\Win)$ depends on $\rho=\frac{d}{N}$ and $w=\frac{K}{N}$ for $N$ large and for $K$ large for the Best-1 case by rewriting the functions $f$ and $\sigma$ for large $N$ as integrals. We constrain $K>>1$,$N>>1$, and $N>K$. 

The function $f$ can be reduced to a new function $F$ for large $N$. We let $w=\frac{K}{N}$, $x=\frac{a}{N}$, $\rho=\frac{d}{N}$ and $z=\frac{m}{N}$. As in the proof of Theorem \ref{ThmM}, because $K$ is large, we define $F_{q}$ with integrals:
\begin{equation} \label{EqF1}
F_{q}(x)=\int\limits_{\rho+(q-1)w}^{x}  \left(\frac{1}{z+w}\right)  \left(1 - \sum_{r=0} ^{q-2} F_{r}(\rho+r w) - F_{q-1}(z-w) \right) dz.
\end{equation}

Similarly, the function $\sigma$ can be reduced to a function $\tau$ for large $N$ with the same normalization, where now $v=\frac{j}{N}$. Then,
\begin{equation} \label{EqT1}
\tau_{q}(x)=\tau_{q-1}(\rho+(q-1) w)+\int\limits_{\rho+(q-1) w} ^{x} \left(1-\sum_{r=0}^{q-2}  F_{r}(\rho+r w) - F_{q-1}(v-w)\right) dv .
\end{equation}

The expressions for $F$ and $\tau$ are functions of solely $w$, $x$ and $\rho$. Thus, $\Pr(\Win)$ is only a function of $w$ and $\rho$, because $\Pr(\Win)=\tau_{q'}(1)$. As a result, the optimal normalized threshold, $\rho^* =\frac{d^*}{N}$, and the optimal probability of winning depend on only $w=\frac{K}{N}$. We now look at the asymptotics for large $K$ and large $N$.
\begin{figure} [h]
\centering
\includegraphics[width=\textwidth]{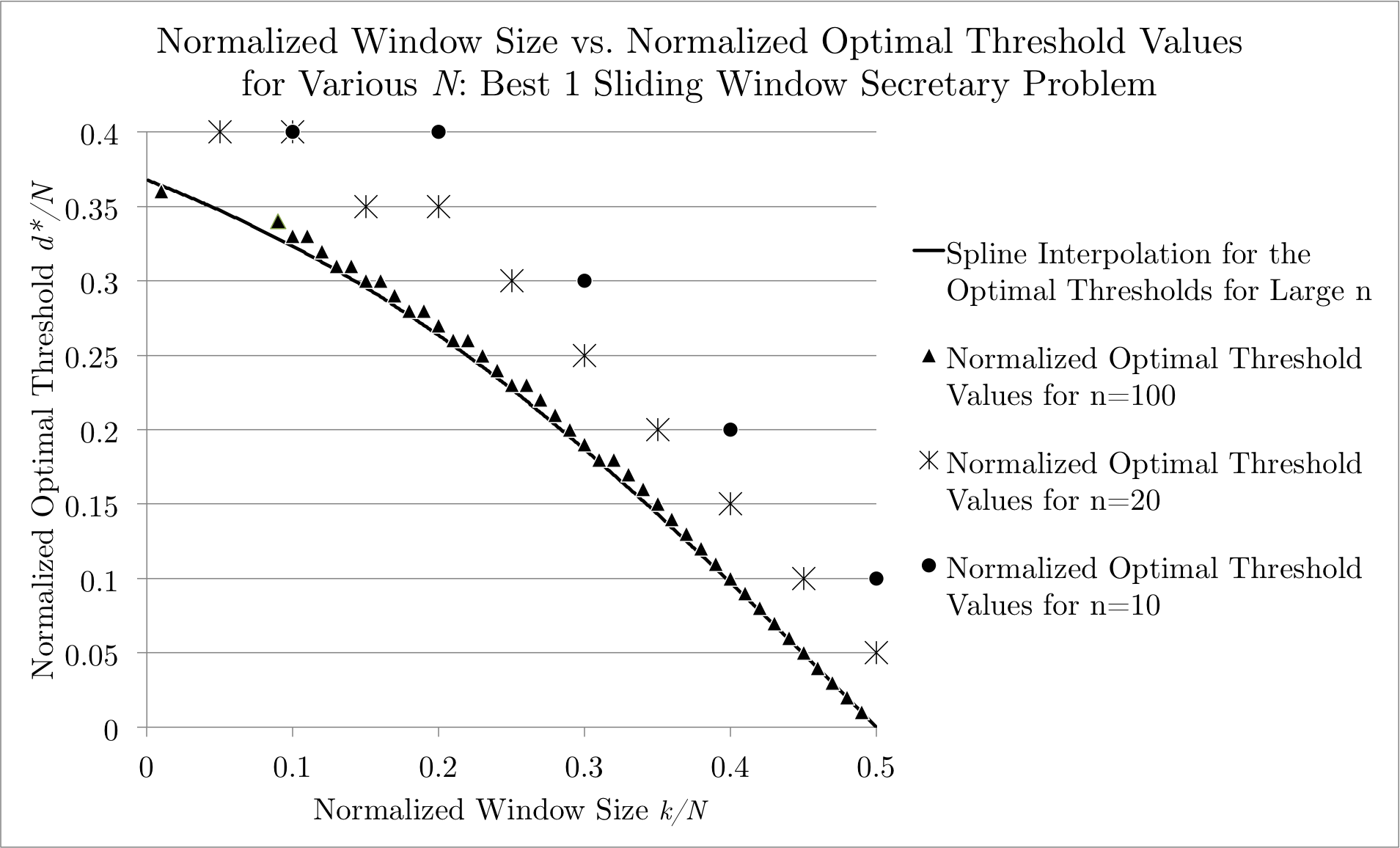}
\caption{How the Normalized Threshold $\frac{d^*}{N}$ varies with Normalized Window Size $\frac{K}{N}$ }
\label{Fig}
\end{figure}

\subsubsection{Asymptotic Optimal Thresholds for Large $N$, Fixed Ratio $\frac{K}{N}$ } \label{lastsection}

We use the definitions of functions $F$ and $\tau$ in Equations \eqref{EqF1} and \eqref{EqT1} respectively to find the optimal $\rho^{*}$ for various $w$. \ref{Ap3} shows some values of $\rho^*$ for a normalized window size $w$ in Table \ref{k3} for large $N$, and Figure \ref{Fig} shows a spline interpolation of the values in Table \ref{k3}, along with values of $\frac{d^{*}}{N}$ for $N=10$, $20$, and $100$. Because the spline interpolation of the values in Table \ref{k3} estimates the optimal thresholds for $N=100$ well, the spline interpolation can predict optimal thresholds for large $N$. In Figure \ref{Fig4}, we show values of the optimal $\Pr(\Win)$ for select values of $\frac{K}{N}$ and $N$ large. The graph predicts the window sizes needed for various probabilities of success.

\begin{figure} [h]
\centering
\includegraphics[width=\textwidth]{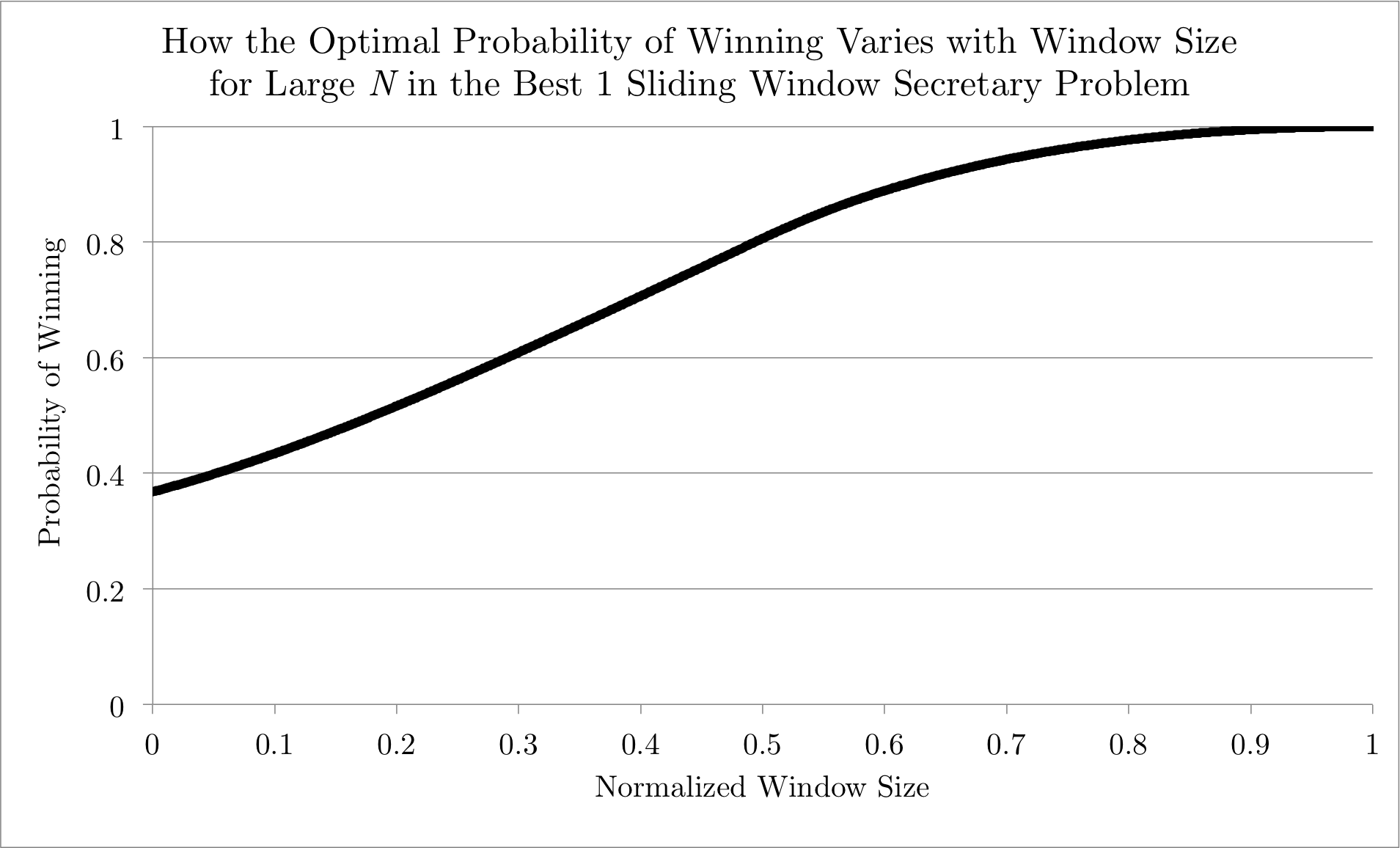}
\caption{The Variation in $\Pr(\Win)$ for Different Values of the Normalized Window Size $w=\frac{K}{N}$ for large $N$ for the Best-1 Sliding-Window Secretary Problem}
\label{Fig4}
\end{figure}
\section{The Sliding-Window Problem: \textit{The Best-2 Case}} \label{Best2}
We now study a Sliding-Window Secretary Problem similar to the Best-1 case, with a payoff of 1 for choosing one of the top two applicants, and 0 otherwise.
A candidate can be the best or second best out of all seen applicants, define a \textit{1-candidate} be a candidate who is the best out of all seen applicants. Define
a \textit{2-candidate} be a candidate who is the second best out of all seen applicants. The interviewer loses nothing if a \textit{sliding rule} is adopted in which the interviewer rejects applicants until the best candidate in the window is at the window's first index. 
We now show that the optimal strategy of the Best-2 case has at most two thresholds.

\begin{theorem} \label{Best2Strategy}
The optimal strategy for the Best-2 case has at most two thresholds, 
$d^*_1$ and $d^*_2$, where the first $d^*_1$ applicants are rejected, 
the first 1-candidate after $d^*_1$ is chosen, and the first 1- or 2-candidate 
after $d^*_2$ is chosen. 
\end{theorem}
\begin{proof}
The probability of winning at an index $i$ given that a 1-candidate is at index $i$ is the probability that the 1-candidate is the best or second best overall. By inclusion-exclusion, 
$$\Pr(\Win \mbox{with 1-Candidate at index} i) = \left(\frac{i+K-1}{N}\right)+ \left(\frac{i+K-1}{N}\right) -\left(\frac{i+K-1}{N}\right)\left(\frac{i+K-2}{N-1} \right).$$

In order for a 2-candidate to be the second best, the best must have been passed,
$$\Pr(\Win \mbox{with 2-candidate at index} i)= \left(\frac{i+K-1}{N}\right) \left( \frac{i+K-2}{N-1} \right).$$

Similar to the proof of Theorem \ref{Cor1}, $\Pr(\Win \mbox{and Reject Candidate } i)$  decreases in $i$. Because applicant $N$ provides a win with probability $\frac{2}{N}$, if applicant $N-K$ is a candidate, the probability of rejecting applicant $N-K$ and winning is $\frac{2}{N}$. A sketch of $\Pr(\Win \mid \mbox{with 1-Candidate } i)$, $\Pr(\Win \mbox{with 2-candidate } i)$, and  $\Pr(\Win \mid \mbox{Rejecting Candidate } i)$ is shown in Figure \ref{T31} to provide intuition for the remaining part of the proof. Therefore, as in Theorem \ref{Cor1}, for both types of candidates we define two integers $d_1^*+1$ and $d_2^*+1$ to be the smallest indices at which the interviewer should not reject a 1-candidate or 2-candidate respectively. Because $\Pr(\Win \mbox{with 1-Candidate } i) \geq \Pr(\Win \mbox{with 2-candidate } i)$, $d^*_1\leq d^*_2$.
\end{proof}

\begin{figure} [h]
\centering
\includegraphics[width=\textwidth]{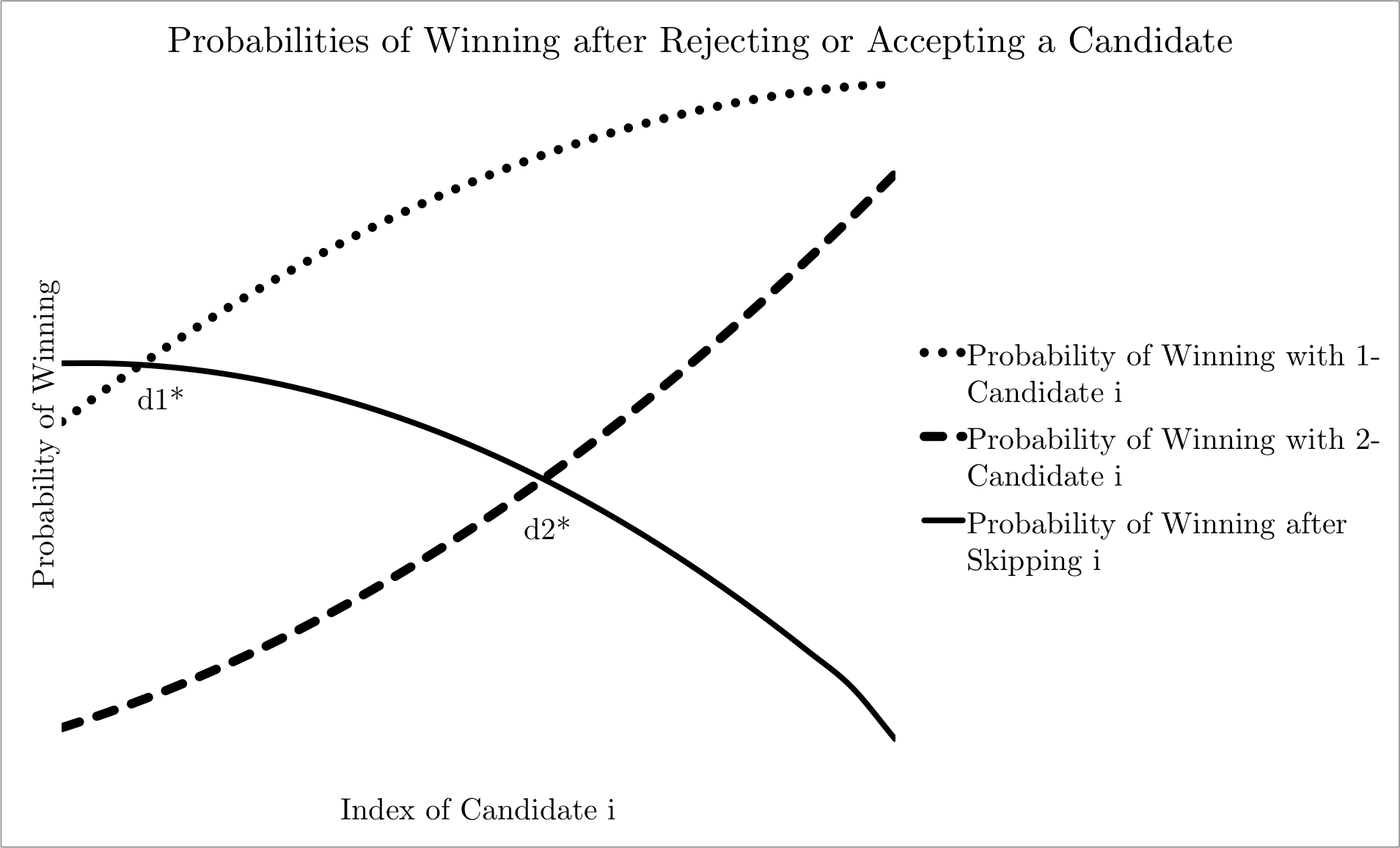}
\caption{A Pictorial Representation of Theorem \ref{Best2Strategy}. The probability of winning with candidate $i$ strictly increases in $i$ for both types of candidates  while the probability of winning with candidate $i$ rejected decreases in $i$. The intersections of the curves show the threshold values.}
\label{T31}
\end{figure}
By the definitions of $d_1^*$ and $d_2^*$ in Theorem \ref{Best2Strategy}, the \textit{optimal strategy for the Best-2 case} is to reject the first $d^*_1$  candidates, to choose the first 1-candidate after $d^*_1$ with the sliding rule, and to choose the first 1- or 2-candidate after $d^*_2$ with the sliding rule.

There are four subcases for the indices of the best and second best applicants, 
$j_1$ and $j_2$, respectively:
\begin{enumerate} [(i)]
\item $j_1\leq d_1$ and $j_2>d_2$
\item $j_2 \leq d_1$ and $j_1>d_1$
\item $j_1>d_1$ and $j_2>d_1$
\item $j_1\leq d_1$ and $j_2\leq d_2$.
\end{enumerate}
Location (iv) guarantees a loss, so we only need to consider (i), (ii), and (iii).

\subsection{Special Cases for $K$}
We first showe $d^*_1 = d^*_2 = 1$ for a window size of $K=N-2$.

For window size $K=N-2$, we find the optimal $d_1^*$ and $d_2^*$ for each of the 
following 4 cases: (i) $d_1=0$ (because if no 1-candidates are skipped we need not consider 2-candidates); (ii) $d_1=1$ and $d_2=1$; (iii) $d_1=1$ and $d_2=2$; and (iV) $d_1=2$ and $d_2=2$. 

For $d_1=0$, the interviewer loses if and only if there is a 1-candidate in the first index, but the second best and best applicants are in the last two indices, i.e., $R(1)=3$, $R(N-1)$ is either 2 or 1, and $R(N)$ is either 1 or 2. This event occurs with probability $\frac{2}{N(N-1)(N-2)}$. For $d_1=1$ and $d_2=1$, the interviewer loses if and only if applicant 1 is either the best or second best, and applicant 2 is a candidate that does not provide a win, i.e., $R(1)$ is either 1 or 2, $R(2)=3$, and $R(N)$ is either 2 or 1. This event occurs with probability $\frac{2}{N(N-1)(N-2)}$. For $d_1=1$, and $d_2=2$, the interviewer loses if and only if the interviewer skips the first and second place applicants, i.e., $R(1)=1$ and $R(2)=2$. This event occurs with probability $\frac{1}{N(N-1)}$. Finally for $d_1=2$ and $d_2=2$, the interviewer loses if and only if the interviewer skips the first and second place applicants, i.e., $R(1)$ is either 1 or 2, and $R(2)$ is either 2 or 1. This event occurs with probability $\frac{2}{N(N-1)}$. Thus, the probability of winning is maximized for $d_1^*=0$ or $d_1^*=1$ and $d_2^*=1$ for $n>4$. 

The results of a simulation for small values of $N$ and $K$ in \ref{Ap4} suggest that as $K$ increases, $\Pr(\Win)$ strictly increases and both optimal thresholds decrease. We formally prove this below. We first prove that $\Pr(\Win)$ strictly increases as $K$ increases.

\begin{lemma} \label{Lemma4}
Let $d^*_{K 1}$ and $d^*_{K 2}$ denote the optimal first and second thresholds respectively for a sliding window of size $K$, and let similar notation hold for $\kappa$. If $\kappa<K$, then $\Pr(\Win \mid \kappa, d^*_{\kappa 1}, d^*_{\kappa 2})<P(\Win \mid K, d^*_{K 1}, d^*_{K 2})$. 
\end{lemma}
\begin{proof}
We prove that $\Pr(\Win \mid \kappa, d^*_{\kappa 1}, d^*_{\kappa 2}) < \Pr(\Win \mid K, d^*_{\kappa 1}, d^*_{\kappa 2}) \leq \Pr(\Win \mid K, d^*_{K 1}, d^*_{K 2})$. Because $d^*_{K 1}$, and $d^*_{K 2}$ are optimal for $K$, $\Pr(\Win \mid K, d^*_{\kappa 1}, d^*_{\kappa 2}) \leq \Pr(\Win \mid K, d^*_{K 1}, d^*_{K 2})$. Therefore, we prove that $\Pr(\Win \mid \kappa, d^*_{\kappa 1}, d^*_{\kappa 2}) < \Pr(\Win \mid K, d^*_{\kappa 1}, d^*_{\kappa 2})$.

Let us first consider winning with a 2-candidate after $d^*_{\kappa 2}$ given that $j_1<d^*_{\kappa 1}$. Since $K>\kappa$, it follows from Lemma \ref{Lemma2} that we win more frequently with a window of $K$ than with a window of $\kappa$. Similarly, let us consider winning with a 1-candidate, or with threshold $d^*_{\kappa 1}$. Then Lemma \ref{Lemma2} exactly applies, either if $j_2<j_1$ or $j_2>j_1$ since both will be considered as 1-candidates. Because each subcase occurs with the same probability for identical thresholds, Simpson's paradox does not apply and therefore $\Pr(\Win \mid \kappa, d^*_{\kappa 1}, d^*_{\kappa 2}) < \Pr(\Win \mid K, d^*_{\kappa 1}, d^*_{\kappa 2})$.
\end{proof}

We now prove that both optimal thresholds decrease as the window size increases.

\begin{lemma} \label{Lemma5}
If $\kappa<K$, then $d^*_{K 1} \leq d^*_{\kappa 1}$ and $d^*_{K 2} \leq d^*_{\kappa 2}$.
\end{lemma}

\begin{proof}
The $d_2$ threshold is only relevant for the case in which $j_1\leq d_1$ and $j_2>d_2$. We can use the same argument as in Lemma \ref{Lemma3} to conclude that $d^*_{K 2} \leq d^*_{\kappa 2}$. Similarly we can use the same argument as in Lemma \ref{Lemma3} for the $d_1$ threshold to conclude that $d^*_{K 1} \leq d^*_{\kappa 1}$.
\end{proof}

We now find the maximum value of $K$ for which $d_2^*$ is no longer 1.
\begin{theorem}
For large $N$, $K=\frac{N-1}{\sqrt 2}$ is the largest window size for which $d^*_2>1$.
\end{theorem}
\begin{proof}
We only consider 2-candidates after $d^*_2$. Because for large $N$, the probability that a 2-candidate in the first $\frac{N}{2}$ applicants is the 2nd best applicant overall is $\frac{1}{4}$, but the probability that there is a better applicant later is $\frac{3}{4}$, we consider only $K>\frac{N}{2}$. Thus the interviewer sees all applicants if he skips a 2-candidate at index $2$.

If $d^*_2$ is 1, $\Pr(\Win)$ is higher if a 2-candidate at index $2$ is chosen than if the 2-candidate is rejected. The probability that the 2-candidate at index $2$ is the second best overall is the probability that the first and second best applicants overall are between $1$ and $K+1$. The probability of winning if the 2nd applicant is rejected is equal to the probability that the best or second best applicants are after $K+1$. Thus, for $d^*_2=1$,
$$\left(\frac{K+1}{N} \right) \left(\frac{K}{N-1}\right) >\left(\frac{N-K-1}{N}\right) +  \left(\frac{K}{N}\right) \left( \frac{N-K-1}{N-1}\right).$$

The largest value at which the inequality does not hold is
\begin{equation*}
K=\frac{\sqrt{8(N-1)^2+1}+1}{4} \approx \frac{N-1}{\sqrt{2}}. \qedhere
\end{equation*}
\end{proof}

\subsection{A Recursive Formula for the Probability of Winning}
As in the Best-1 case, we now derive a general solution for the Best-2 case, assuming possibly non-optimal thresholds of $d_1$ and $d_2$. We again consider blocks of size $K$ after $d_1$ and after $d_2$. Let $h_{s}(a)$ be the probability of stopping at a 2-candidate between $(d_2+(s-1)K+1)$ and $a$ given that the best applicant that the interviewer has interviewed is between $1$ and $d_1$, where $s=\lceil\frac{a-d_2}{K}\rceil$. Let $g_{s}(a)$ be the probability of stopping at a 2-candidate between $d_2+(s-1)K+1$ and $a$ where $s=\lceil\frac{a-d_2}{K}\rceil$. Finally, let $f_{q}(a)$ be the probability that the interviewer stops at a 1-candidate between $d_1 +(q-1)K+1$ and $a$, where $q=\lceil\frac{a-d_1}{K}\rceil$. For $q\leq 0$, $f$ is 0, and for $s\leq 0$, $g$ and $h$ are both 0, because the interviewer does not select 2-candidates before $d_2$ and does not select 1-candidates before $d_1$. We now present recursive formulas for $f_{q}$, $g_{s}$, and $h_{s}$, for $q>0$ and for $s>0$. For convenience, let 
$$c(i)=\sum_{r=-1}^{\lceil\frac{i-d_2}{K}\rceil-2}(h_r(d_2+rK))+ h_{\lceil\frac{i-d_2}{K}\rceil-1}(i-K),$$
and
$$t(i)=\sum_{r=-1}^{\lceil\frac{i-d_1}{K}\rceil-2}(f_r(d_1+rK)) + f_{\lceil\frac{i-d_1}{K}\rceil-1}(i-K) +\sum_{r=-1}^{\lceil\frac{i-d_2}{K}\rceil-2}(g_r(d_2+rK)) + g_{\lceil\frac{i-d_2}{K}\rceil-1}(i-K).$$

\begin{lemma} \label{importantt}
For $s>0$ and for $q>0$, 
$$h_s(a)=\sum_{i=d_2+(s-1)K+1}^{a} \left(\frac{1}{i+K-2}\right) \left(1-c(i) \right),$$
$$g_s(a)=\sum_{i=d_2+(s-1)K+1}^{a} \left(\frac{d_1}{i+K-1}\right) \left( \frac{1}{i+K-2}\right) \left(1-c(i)\right),$$
$$f_q(a)=\sum_{i=d_1+(q-1)K+1}^{a} \left( \frac{1}{i+K-1}\right) \left(1-t(i) \right).$$
\end{lemma}

\begin{proof}
For the function $h$, the probability of stopping at a 2-candidate at an index $i$ given that the best applicant out of all seen applicants is in $[1,d_1]$ is $\frac{1}{i+K-2}$, because 1 of the $(i+K-1)$ indices is occupied by the best applicant so far. In addition, probabilities of stopping earlier can be subtracted in blocks of $K$ as in the proof of Lemma \ref{importantthm}. Therefore, 
$$h_s(a)=\sum_{i=d_2+(s-1)K+1}^{a} \left(\frac{1}{i+K-2}\right) \left(1-c(i) \right).$$

For the function $g$, an additional term $\frac{d_1}{i+K-1}$ is added to guarantee that the best applicant out of all seen applicants is in the first $d_1$ indices. When subtracting the probabilities of stopping before, we use the function $h$ because the additional term already accounts for the best applicant so far being restricted to $[1,d_1]$. Therefore, 
$$g_s(a)=\sum_{i=d_2+(s-1)K+1}^{a} \left( \frac{d_1}{i+K-1} \right) \left( \frac{1}{i+K-2}\right) \left(1-c(i) \right).$$

Finally, for the function $f$, the probability of having a 1-candidate at an index $i$ is $\frac{1}{i+K-1}$, while the probabilities of stopping at earlier candidates need to be subtracted. Therefore,
\begin{equation*}
f_q(a)=\sum_{i=d_1+(q-1)K+1}^{a} \left( \frac{1}{i+K-1}\right) \left(1-t(i) \right). \qedhere
\end{equation*}
\end{proof}

Now, let $\sigma_1(a)$ return the probability of winning with a candidate between 1 and $a$ for subcase (i). Let $\sigma_2(a)$ return the probability of winning with a candidate between 1 and $a$ for subcase (ii). Let $\sigma_3(a)$ return the probability of winning with a candidate between 1 and $a$ for subcase (iii). For $a\leq d_2$, $\sigma_1(a)=0$, because we do not choose 2-candidates before $d_2$. For $a\leq d_1$, $\sigma_2(a)= \sigma_3(a)=0$ because we do not choose candidates before $d_1$. We now present recursive formulas for the $\sigma$ functions.

\begin{theorem} \label{Best2Thm}
Let $q=\lceil\frac{a-d_1}{K}\rceil$ and $s=\lceil \frac{a-d_2}{K}\rceil$. Then,
$$\sigma_1(a)=\sigma_1(d_2+(s-1)K) + \sum_{i=d_2+(s-1)K+1}^{a} \frac{d_1}{N}\left( \frac{1}{N-1}\right) \left(1-c(i)\right),$$
$$\sigma_2(a)=\sigma_2(d_2+(s-1)K) + \sum_{i=d_2+(s-1)K+1}^{a} \frac{d_1}{N}\left( \frac{1}{N-1}\right) \left(1-c(i)\right),$$
$$\sigma_3(a)=\sigma_3(d_1+(q-1)K)+\sum_{i=d_1+(q-1)K+1}^{a} 2 \left( \frac{N-i}{N(N-1)}\right) \left(1-t(i) \right).$$
\end{theorem}

\begin{proof}
We find $\sigma_1(a)$ by summing $\sigma_1(d_2+(s-1)K)$ with the probability of winning with a candidate in $[(d_2+(s-1)K), a]$. The probability of stopping at $j_2$ in subcase (i) is $\frac{d_1}{N}\Big( \frac{1}{N-1}\Big)$, because the best applicant has to be among the first $d_1$ applicants After subtracting  probabilities of stopping earlier, we find 
\begin{equation*}
\sigma_1(a)=\sigma_1(d_2+(s-1)K) + \sum_{i=d_2+(s-1)K+1}^{a} \frac{d_1}{N}\left( \frac{1}{N-1}\right) \left(1-c(i)\right).
\end{equation*}

The probability of stopping at $j_1$ in subcase (ii) similarly is $\frac{d_1}{N}\left( \frac{1}{N-1}\right)$. Therefore,
\begin{equation*}
\sigma_2(a)=\sigma_2(d_2+(s-1)K) + \sum_{i=d_2+(s-1)K+1}^{a} \frac{d_1}{N}\left( \frac{1}{N-1}\right) (1-c(i)).
\end{equation*}

Finally, the probability of stopping at either $j_1$ or $j_2$ in subcase (iii) is $2 \Big( \frac{N-i}{N(N-1)}\Big)$, because permuting $j_1$ and $j_2$ does not affect $\Pr(\Win)$. Accounting for not stopping earlier yields
\begin{equation*}\sigma_3(a)=\sigma_3(d_1+(q-1)K)+\sum_{i=d_1+(q-1)K+1}^{a} 2 \Big( \frac{N-i}{N(N-1)}\Big) \left(1-t(i) \right). \qedhere
\end{equation*}
\end{proof}
We now find that $\Pr(\Win)=\sigma_1(n)+\sigma_2(n) + \sigma_3(n)$.

For large $N$ and large $K$, we normalize the functions so that $H\left(\frac{a}{N}\right) \approx h(a)$, $G\left(\frac{a}{N}\right) \approx g(a)$, $F\left(\frac{a}{N}\right) \approx f(a)$, and $\tau_l(\frac{a}{N})\approx \sigma_l(a)$ for $l \in \{1,2,3\}$. We let $w=\frac{K}{N}$, $x=\frac{a}{N}$, $\rho_1=\frac{d_1}{N}$, and $\rho_2=\frac{d_2}{N}$, and write the asymptotic functions using integrals:
$$H_s(x)=\int\limits_{\rho_2+(s-1)w}^{x}\left(\left( \frac{1}{v+w} \right) \left(1-C(v) \right)\right) dv,$$ 
$$G_s(x)=\int\limits _{\rho_2+(s-1)}^{w} \left( \left( \frac{\rho_1}{v+w} \right) \left( \frac{1}{v+w}\right) \left(1-C(v) \right)\right) dv,$$
$$F_q(x)=\int\limits_{\rho_1+(s-1)w}^{x} \left( \left( \frac{1}{v+w}\right) \left(1-T(v) \right) \right) dv,$$
where
$$C(v)=\sum_{r=-1}^{\lceil \frac{v-\rho_2}{w}\rceil-2}(H_r(\rho_2+rw))+ H_{\lceil \frac{v-\rho_2}{w}\rceil-1}(v-w),$$
and
$$T(v)=\sum_{r=-1}^{\lceil \frac{v-\rho_2}{w}\rceil-2}(F_r(\rho_1+rw)) + F_{\lceil \frac{v-\rho_1}{w}\rceil-1}(v-w) +\sum_{r=-1}^{\lceil \frac{v-\rho_2}{w}\rceil-2}(G_r(\rho_2+rw)) + G_{\lceil \frac{v-\rho_2}{w}\rceil-1}(v-K).$$
Similarly,
$$\tau_1(x)=\tau_1(\rho_2+(s-1)w) + \int\limits_{\rho_2+(s-1)w}^{x} \left(\rho_1 (1-C(v))\right) dv,$$
$$\tau_2(x)=\tau_2(\rho_2+(s-1)w) + \int\limits_{\rho_2+(s-1)w}^{x} \left(\rho_1 (1-C(v)) \right) dv,$$
$$\tau_3(x)=\tau_3(\rho_1+(q-1)w)+\int\limits_{\rho_1+(q-1)w}^{x} \left(2(1-v) (1-T(v))\right) dv.$$ 

As with the Best-1 case, $\Pr(\Win)=\tau_1(1)+\tau_2(1)+\tau_3(1)$ depends only on $w$, $\rho_1$, and $\rho_2$ for large $N$. When finding the optimal $\rho_1$ and $\rho_2$, or $\rho_1^*$ and $\rho_2^*$, the two systems of equations $\frac{d\Pr(\Win)}{d\rho_1}=0$ and $\frac{d\Pr(\Win)}{d\rho_2}=0$, are solved, and thus $\rho_1^*$, $\rho_2^*$, and the optimal $\Pr(\Win)$ only depend on $\frac{K}{N}$ for $N$ large. We now look at asymptotics for large $K$ and large $N$.

\subsection{Asymptotic Optimal Thresholds: Large $N$, Fixed $\frac{K}{N}$ }

As with the Best-1 case, we can use the recursions for large $N$ to find the optimal normalized thresholds $\rho_1^*$ and $\rho_2^*$ as a function of the normalized window size $w$. However, MATLAB was not able to compute the recursions in integral form. As a result, we present the optimal normalized thresholds for the normalized window size $w$, and three cases: $N=10$, $N=100$, $N=1000$. A spline interpolation of the optimal normalized thresholds for $N=1000$, along with the values of the optimal normalized thresholds for $N=10$ and $N=100$ for select normalized window sizes is shown in Figure \ref{Figg}.
\begin{figure} [h]
\centering
\includegraphics[width=\textwidth]{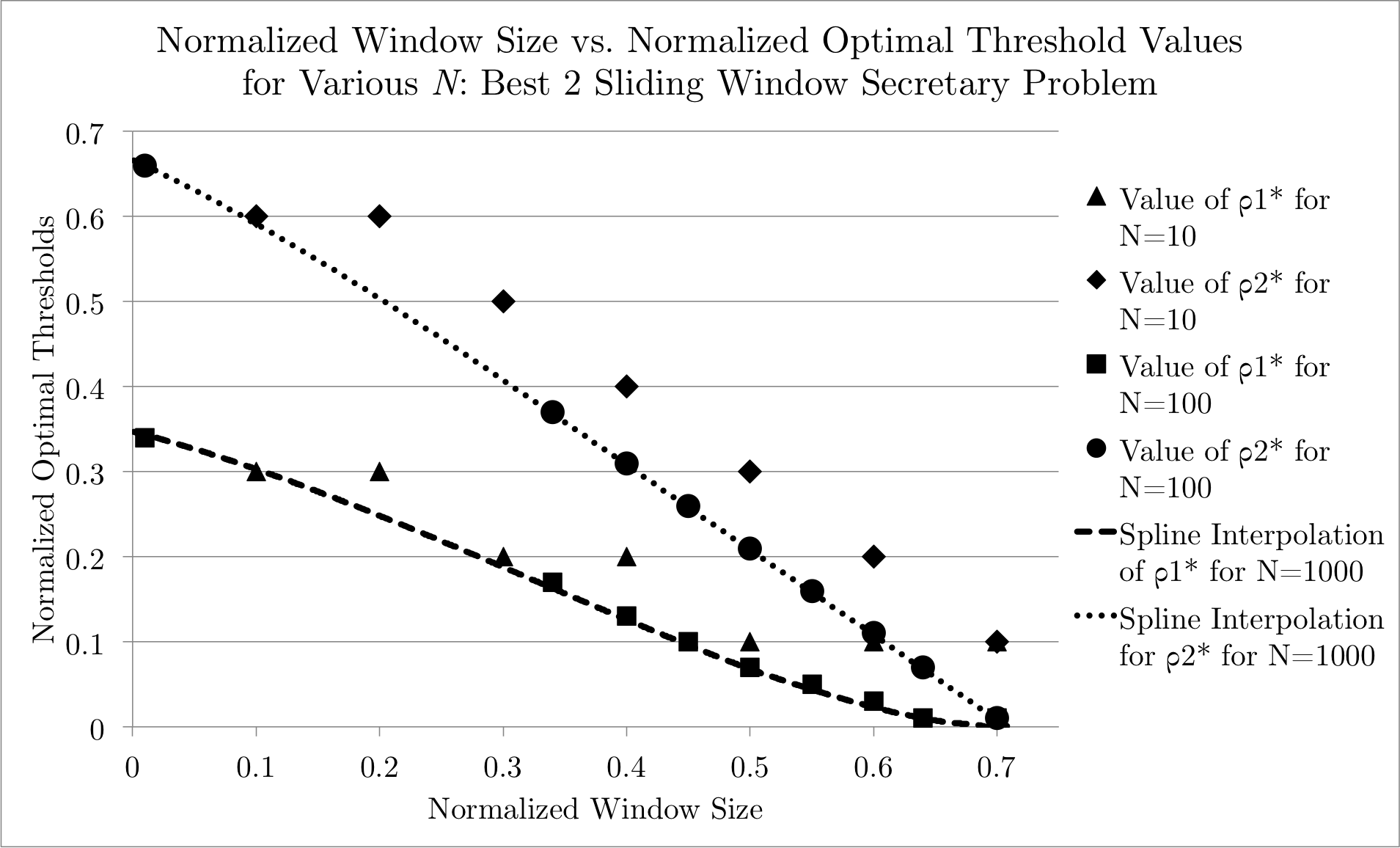}
\caption{The Variation of the Normalized Thresholds $\rho_1^{*} =\frac{d^*_1}{N}$ and $\rho_2^*=\frac{d^*_2}{N}$ varies with the Normalized Window Size $w=\frac{K}{N}$ for $N=10$, $20$, $100$}
\label{Figg}
\end{figure}

In addition, a spline interpolation of select values of $K$ and the probability of winning with that value of $K$ is shown in Figure \ref{Figg2} for $N=100$ to show what values of $K$  are required to guarantee certain probabilities of winning.if an applicant is a candidate.

\begin{figure} [h]
\centering
\includegraphics[width=\textwidth]{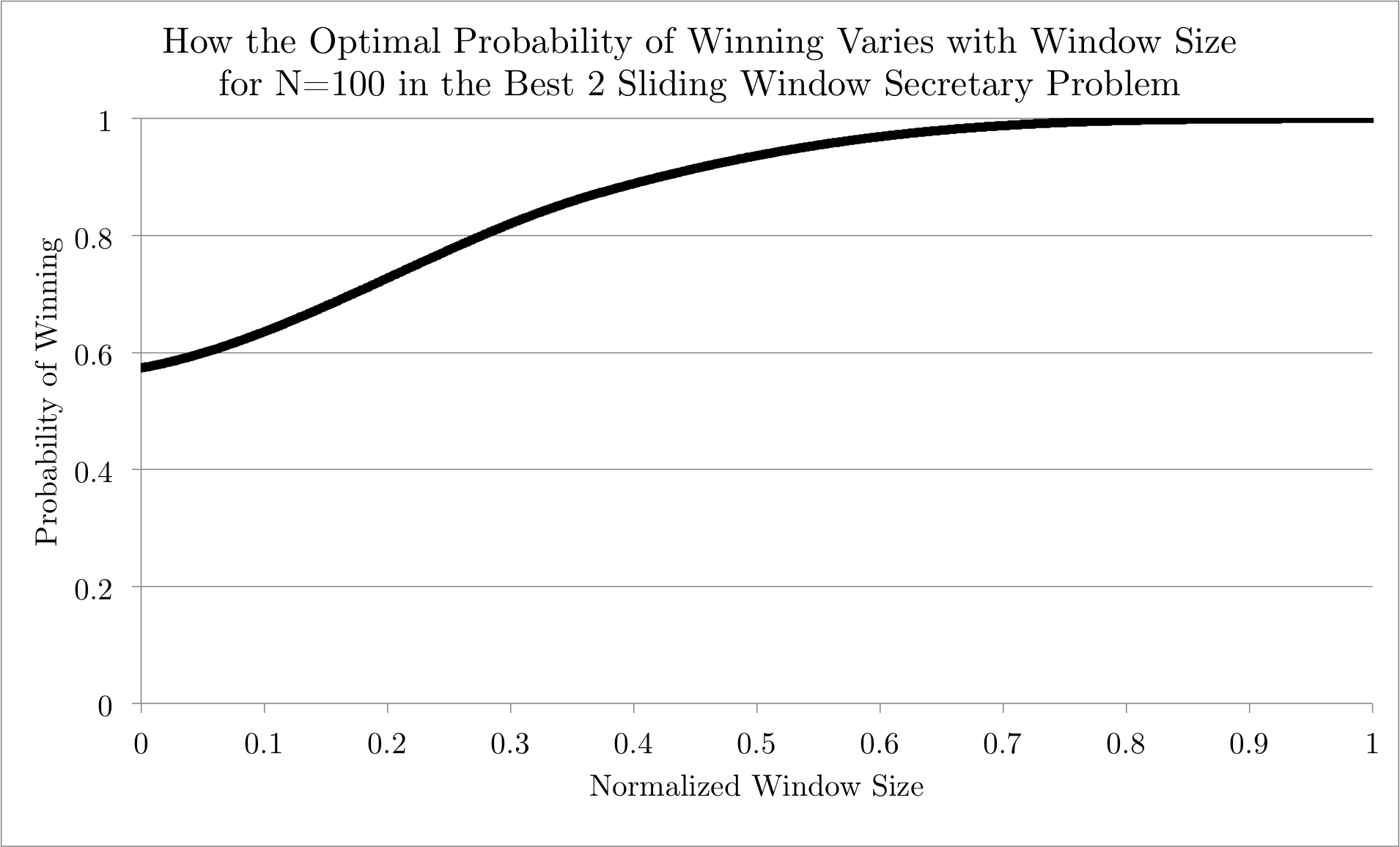}
\caption{The Variation in $\Pr(\Win)$ for Different Values of the Normalized Window Size $x=\frac{K}{N}$ for large $N$ for the Best-2 Sliding-Window Secretary Problem}
\label{Figg2}
\end{figure}

\subsection{Extensions to Winning with One of the Top $L$ Applicants}
Our analysis generalizes in a straightforward way to the top $L$ case, where
the interviewer wins if one of the Top $L$ applicants is chosen. Because there will be
$L$ different types of candidates in the Top-$L$ problem, and it follows from Theorem \ref{Best2Strategy} that the optimal strategy has $L$ thresholds. Equations similar to those in Theorem \ref{Best2Thm} can be used to find the probability of winning for various window sizes, some fixed number of applicants, and different threshold values. 


\section{The Sliding-Window Problem: The \textit{2-Choice Case}} \label{2Choice}
We now examine a Sliding-Window Secretary Problem similar to the Best-1 Case, where we grant the interviewer the ability to choose two applicants and a win occurs if either of the two applicants is the best overall. Again, the same \textit{sliding rule} is implemented because
it costs us nothing. We first show that the 2-Choice Case has two thresholds and then the
optiomal decision strastegy given the two thresholds.

\begin{theorem} \label{2ChoiceStrategy}
The optimal strategy has most 2 thresholds for the 2-Choice Case.
Given thresholds $\delta^*_1$ and $\delta^*_2$, the optimal strategy is to reject the first $\delta^*_1$ applicants, to choose the first candidate, $m_1$, after $\delta^*_1$ with the sliding rule after $\delta^*_1$, and then to choose the first candidate, $m_2$, after both $m_1$ and $\delta^*_2$.
\end{theorem}
\begin{proof}
We first prove that the interviewer's second choice has an optimal threshold, $\delta^*_2$. By the sliding rule,
$$\Pr(\Win \mbox{with candidate $i$ as a second choice})= \frac{i+K-1}{N}.$$
As with the proof to Theorem \ref{Cor1}, the probability of rejecting candidate $i$ and winning decreases in $i$, and is lower than the probability of choosing candidate $i$ and winning if $i=N-K$. Therefore, as with the proof to Theorem \ref{Cor1}, for some optimal threshold $\delta^*_2$, choosing the first candidate after $\delta^*_2$ as a second choice maximizes the probability of winning.

We now prove that the interviewer's first choice has an optimal threshold, $\delta^*_1$. We consider the function $q(i)$, the probability of winning if the interviewer's first choice is candidate $i$, and $p(i)$, the probability of choosing the best applicant as a second choice given that the best applicant is after candidate $i$. Because the interviewer wins either if candidate $i$ is the best applicant, or if the interviewer's second choice after $i$ is the best applicant,
$$q(i) = \frac{i+K-1}{N} + \left(1-\frac{i+K-1}{N}\right) p(i).$$
If the best applicant is after index $i$, the interviewer finds the best applicant with higher probability if there are fewer applicants after $i$ and thus if $i$ is larger. Therefore $p(i)$ increases in $i$. We show that $q(i)$ increases in $i$ by computing $q(i+1)-q(i)$:
$$q(i+1)-q(i)=\frac{1-p(i)}{N}+\left(1-\frac{i+K}{N}\right)(p(i+1)-p(i)).$$
Because $p(i+1) \leq p(i)$ and $p(i)\leq 1$, $q(i+1)-q(i)\geq 0$ and $q(i)$ increases in $i$. Additionally, as with the proof to Theorem \ref{Cor1}, the probability of rejecting candidate $i$ and winning decreases in $i$, and is lower than the probability of choosing candidate $i$ and winning if $i=N-K$. Therefore, as with the proof to Theorem \ref{Cor1}, for some optimal threshold $\delta^*_1$, choosing the first candidate after $\delta^*_1$ as a second choice maximizes the probability of winning.
\end{proof}

Let $\delta^*_1$ and $\delta^*_2$ be as defined in Theorem \ref{2ChoiceStrategy}. Then it is optimal to reject all applicants before $\delta^*_1+1$, to choose the first candidate to appear after $\delta^*_1$ with the sliding rule, and to choose another candidate who is the first to appear after both the first candidate and $\delta^*_2$. Note that $\delta^*_2 \geq \delta^*_1 +K$ because a block of $K$ applicants cannot have two candidates by the sliding rule. The interviewer can win in two mutually exclusive subcases:
\begin{enumerate} [(i)]
\item The first choice is the winning choice
\item The second choice is the winning choice and the first choice was made before $\delta^*_2-K+2$.
\item The second choice is the winning choice and the first choice was made after $\delta^*_2-K+1$. 
\end{enumerate}

\subsection{Special Cases for $K$}
We first find that the optimal probability of winning is $1$ for a window size of $K\geq \frac{N}{2}$.

\begin{theorem}
The optimal probability of winning for the interviewer is $1$ if $K \geq \frac{N}{2}$. 
\end{theorem}
\begin{proof}
Because there cannot be two candidates in the same block of $K$ applicants by the sliding rule, and because $2K>N$, there are only at most two candidates that the interviewer can consider. Because the interviewer has two choices, the interviewer can choose both candidates. Because the best applicant overall is guaranteed to be a candidate, the interviewer is guaranteed to win. 
\end{proof}

For the interviewer to choose all possible candidates for a window size of $K \geq \frac{N}{2}$, $\delta_1=0$ and $\delta_2=K$. Therefore, the optimal thresholds $\delta^*_1$ and $\delta^*_2$ are $0$ and $K$ respectively.

We now prove that as $K$ increases, $\Pr(\Win)$ strictly increases for fixed $N$.

\begin{lemma} \label{Lemma6}
Let $\delta^*_{K1}$ and $\delta^*_{K2}$ be the optimal first and second thresholds respectively for a window size of $K$, and let the same notation hold for a window size of $\kappa$. Then if $\kappa>K$, $\Pr(\Win \mid \kappa, \delta^*_{\kappa 1}, \delta^*_{\kappa 2})<\Pr(\Win \mid K, \delta^*_{K 1}, \delta^*_{K2})$. 
\end{lemma}
\begin{proof}
We prove that $\Pr(\Win \mid \kappa, \delta^*_{\kappa 1}, \delta^*_{\kappa 2})< \Pr(\Win \mid K, \delta^*_{\kappa 1}, \delta^*_{\kappa 2}) \leq \Pr(\Win \mid K, \delta^*_{K 1}, \delta^*_{K2})$. Since $\delta^*_{K 1}$, and $\delta^*_{K2}$ are optimal for $K$, $\Pr(\Win \mid K, \delta^*_{\kappa 1}, \delta^*_{\kappa 2}) \leq \Pr(\Win \mid K, \delta^*_{K 1}, \delta^*_{K2})$, and we now prove that $\Pr(\Win \mid \kappa, \delta^*_{\kappa 1}, \delta^*_{\kappa 2})< \Pr(\Win \mid K, \delta^*_{\kappa 1}, \delta^*_{\kappa 2})$. 

By the same argument as in Lemma \ref{Lemma2}, every sequence of applicants that produces a win with thresholds $\delta^*_{\kappa 1}$, and $\delta^*_{\kappa 2}$, and a window size of $\kappa$ must also produce a win for identical thresholds and a window size of $K$. As in Lemma \ref{Lemma2}, we can construct a sequence of applicants such that the interviewer loses with thresholds $\delta^*_{\kappa 1}$, and $\delta^*_{\kappa 2}$ and a window size of $\kappa$, but wins with identical thresholds and a window size of $K$. Therefore, $\Pr(\Win \mid \kappa, \delta^*_{\kappa 1}, \delta^*_{\kappa 2})< \Pr(\Win \mid K, \delta^*_{\kappa 1}, \delta^*_{\kappa 2})$. 
\end{proof}

We now prove that as $K$ increases, the first optimal threshold decreases for fixed $N$. The second optimal threshold must be at least $K$ greater than the first optimal threshold so the second threshold may not necessarily decrease as $K$ increases. 
\begin{lemma} \label{Lemma7}
If $\kappa>K$, then $\delta^*_{\kappa 1}\geq \delta^*_{K1}$. 
\end{lemma}
\begin{proof}
From the proof of Theorem \ref{2ChoiceStrategy}, for a candidate $i$, $\Pr(\Win \mid \mbox{$i$ is first choice, } K)>\Pr(\Win \mid \mbox{$i$ is first choice, } \kappa)$. Using a similar argument as in Lemma \ref{Lemma3}, for a candidate $i$, $\Pr(\Win \mid \mbox{Rejecting $i$, } K)>\Pr(\Win \mid \mbox{Rejecting $i$, } \kappa)$ for the first choice. By the same argument as in Lemma \ref{Lemma3},  $\delta^*_{\kappa 1}\geq \delta^*_{K1}$.
\end{proof}

\subsection{A Recursive Formula for the Probability of Winning}
We now analyze the problem for some window size $K$ and some threshold values $\delta_1$ and $\delta_2$. We again will use blocks of size $K$ because of the sliding rule. We modify the function $f$ in Lemma \ref{importantthm} as follows: we let $f(x,a)$ be the probability of choosing an applicant between indices $x+K\left(\lceil\frac{a-x}{K}\rceil-1\right)+1$ and $a$. For $a\leq x$, $f(x,a)=0$. For $a>x$, using the same argument as in Lemma \ref{importantthm}, we find that 
$$f(x,a)=\sum_{m=x+(\lceil\frac{a-x}{K}\rceil-1) K + 1} ^{a} \frac{1}{m+K-1}  \left(1 - \sum_{r=-1} ^{\lceil\frac{a-x}{K}\rceil-2} f(x,x+r K) - f(x,m-K) \right).$$

We now let $g(m,x,b)$ be the probability of making a choice at an index $m$, not making another choice until index $b-k+1$, given that the interviewer is guaranteed not to make a choice before index $x+1$, and the interviewer chooses the applicant at index $b$. Let $c(m,x)$ be the probability of making a choice at index $m$ given that the interviewer is guaranteed not to make a choice before index $x+1$. For $m\leq x$, $c(m,x)=0$, and for $m>x$, 
$$c(m,x)=\frac{1}{m+K-1}  \left(1 - \sum_{r=-1} ^{\lceil\frac{m-x}{K}\rceil-2} f(x,x+r K) - f(x,m-K) \right).$$
By the sliding rule, because the first and second choices cannot be in the same block of $K$ applicants, if $b-m< K$, $g(m,x,b)=0$. Otherwise, because there are no candidates between $m+1$ and $m+K-1$ by the sliding rule if $m$ is a candidate, and a random ordering of applicants guarantees that finding a candidate starting at index $m+K$ is independent of finding a candidate at index $m$, we can multiply $c(m,x)$ by the probability that we find no other candidate after index $m+K-1$, given that applicant $m$ is a candidate that we have chosen. Therefore,
$$g(m,x,b)=c(m,x)\left(1 - \sum_{r=-2} ^{\lceil\frac{b-m-K+1}{K}\rceil-1} f(m+K-1,m+rK-1) - f(m+K-1,b-K)  \right).$$

We now divide the probability of winning as follows: $\sigma_1(a)$ is the probability of winning between $[1,a]$ with subcase (i),  $\sigma_2(a)$ is the probability of winning between $[1,a]$ with subcase (ii), and $\sigma_3(a)$ is the probability of winning between $[1,a]$ with subcase (iii). For $a\leq \delta_1$, $\sigma_1(a)=0$, and for $a\leq \delta_2$, $\sigma_2(a)=\sigma_3(a)=0$. We compute $p_{b}$, the probability of making a choice between $\delta_1+1$ and $\delta_2-K+1$ as 
$$p_{b}=\sum_{r=0}^{\lceil \frac{\delta_2-\delta_1-K+1}{K} \rceil-1} f(\delta_1, \delta_1+rK) + f(\delta_1, \delta_2-K+1).$$

\begin{theorem}
If $a>\delta_1$ and $q=\lceil\frac{a-\delta_1}{K} \rceil$ then
$$\sigma_1(a)=\sigma_{1}(\delta_1+(q-1)K)+\frac{1}{N} \sum_{j=\delta_1+(q-1)K+1} ^{a} \left(1-\sum_{r=-1}^{q-2}  f(\delta_1, \delta_1+r K) - f(\delta_1,j-K)\right).$$
If $a>\delta^*_2$ and $q=\lceil\frac{a-\delta_2}{K} \rceil$ then
$$\sigma_2(a)=\sigma_{2}(\delta_2+(q-1)K)+\frac{p_{b}}{N} \sum_{j=\delta_2+(q-1)K+1} ^{a} \left(1-\sum_{r=-1}^{q-2}  f(\delta_2, \delta_2+r K) - f(\delta_2,j-K)\right).$$
If $a>\delta^*_2$ then
$$\sigma_3(a)=\sigma_{3}(a-1) + \frac{1}{N}\sum_{m=\delta_2-K+2}^{a-1} g(m,\delta_1,a).$$
\end{theorem}
\begin{proof}
The expression for $\sigma_1$ follows directly from Theorem \ref{impttheorem}. Similarly, the expression for $\sigma_2$ follows directly from Theorem \ref{impttheorem}, except with the added condition that a choice is made between $\delta_1+1$ and $\delta_2-K+1$. The probability of making a first choice at an index $m$ and making a second choice at index $j$ is given by $g(m,\delta_1,j-K+1)$, because the interviewer is guaranteed to choose the best applicant at index $j$ by the sliding rule if the interviewer does not make a choice between $m+1$ and $j-K$. Therefore, by the total probability theorem, we can add up the probabilities for all possible values of $m$ and find that 
\begin{equation*}
\sigma_3(a)=\sigma_{3}(a-1) + \frac{1}{N}\sum_{m=\delta_2-K+2}^{a-1} g(m,\delta_1,a-K+1). \qedhere
\end{equation*}
\end{proof}
It follows that $\sigma_1(N)+\sigma_2(N)+\sigma_3(N)=\Pr(\Win)$.
In order to normalize $\sigma_3$, we write $\sigma_3(a)$ as follows, where $q=\lceil \frac{a-\delta_2}{K} \rceil$: 
\begin{equation*}
\sigma_3(a)=\sigma_{3}(\delta_2 + (q-1)K) + \frac{1}{N}\sum_{j=\delta_2+(q-1)K+1}^{a} \left( \sum_{m=\delta_2-K+2}^{j-1} g(m,\delta_1,j-K+1)\right). 
\end{equation*}
We now look at large $K$ and large $N$. We first normalize $f$ as the function $F$ as done in earlier cases, by dividing all indices by $N$ and approximating $f$ with integrals. We similarly normalize $c$ as $C$ and $g$ as $G$. Let $\frac{K}{N}=w$, $\frac{\delta_1}{N}=\rho_1$, $\frac{\delta_2}{N}=\rho_2$, $\frac{a}{N}=\alpha$, $\frac{b}{N}=\beta$, $\frac{x}{N}=\gamma$, $\frac{j}{N}=\eta$, and $\frac{m}{N}=\mu$. Then,
$$F(\gamma,\alpha)= \int \limits_{\gamma + \left(\lceil \frac{\alpha-\gamma}{w}\rceil-1\right)w}^{\alpha} \frac{1}{\mu+w}\left(1 - \sum_{r=-1} ^{\lceil\frac{\alpha-\gamma}{w}\rceil-2} F(\gamma,\gamma+r w) - F(\gamma,\mu-w) \right) d\mu,$$
$$C(\mu,\gamma) = \frac{1}{\mu+w}\left(1 - \sum_{r=-1} ^{\lceil\frac{\alpha-\gamma}{w}\rceil-2} F(\gamma,\gamma+r w) - F(\gamma,\mu-w)\right),$$
$$G(\mu,\gamma,\beta)= C(\mu,\gamma)\left(1-\sum_{r=-2} ^{\lceil\frac{\beta-\mu-w}{w}\rceil-1} f(\mu+w,\mu+rw) - f(\mu+w,\beta-w)\right).$$

We now normalize $\sigma_i$ as $\tau_i$, where $i$ is 1, 2, or 3. We additionally normalize $p_b$ as $P_b$:
$$P_{b}=\sum_{r=0}^{\lceil \frac{\rho_2-\rho_1-w}{w} \rceil-1} F(\rho_1, \rho_1+rw) + F(\rho_1, \rho_2-w).$$
Then we find that if $q_1=\frac{\alpha-\rho_1}{w}$, and $q_2=\frac{\alpha-\rho_2}{w}$, 
$$\tau_1(\alpha)=\tau_{1}(\rho_1+(q_1-1)w)+\int \limits_{\rho_1+(q_1-1)w} ^{\alpha} \left(1-\sum_{r=-1}^{q_1-2}  F(\rho_1, \rho_1+r w) - F(\rho_1,\eta-w)\right) d\eta,$$
$$\tau_2(\alpha)=\tau_{2}(\rho_2+(q_2-1)w)+ P_b\int \limits_{\rho_2+(q_2-1)w} ^{\alpha} \left(1-\sum_{r=-1}^{q_2-2}  F(\rho_2, \rho_2+r w) - F(\rho_2,\eta-w)\right) d\eta,$$
$$\tau_3(\alpha)=\tau_{3}(\rho_1 + (q_2-1)w) + \int \limits_{\rho_2+(q_2-1)w}^{a} \left( \int \limits_{\rho_2-w}^{\eta} G(\mu,\rho_1,\eta-w)\right) d\eta.$$

\begin{figure} [h]
\centering
\includegraphics[width=\textwidth]{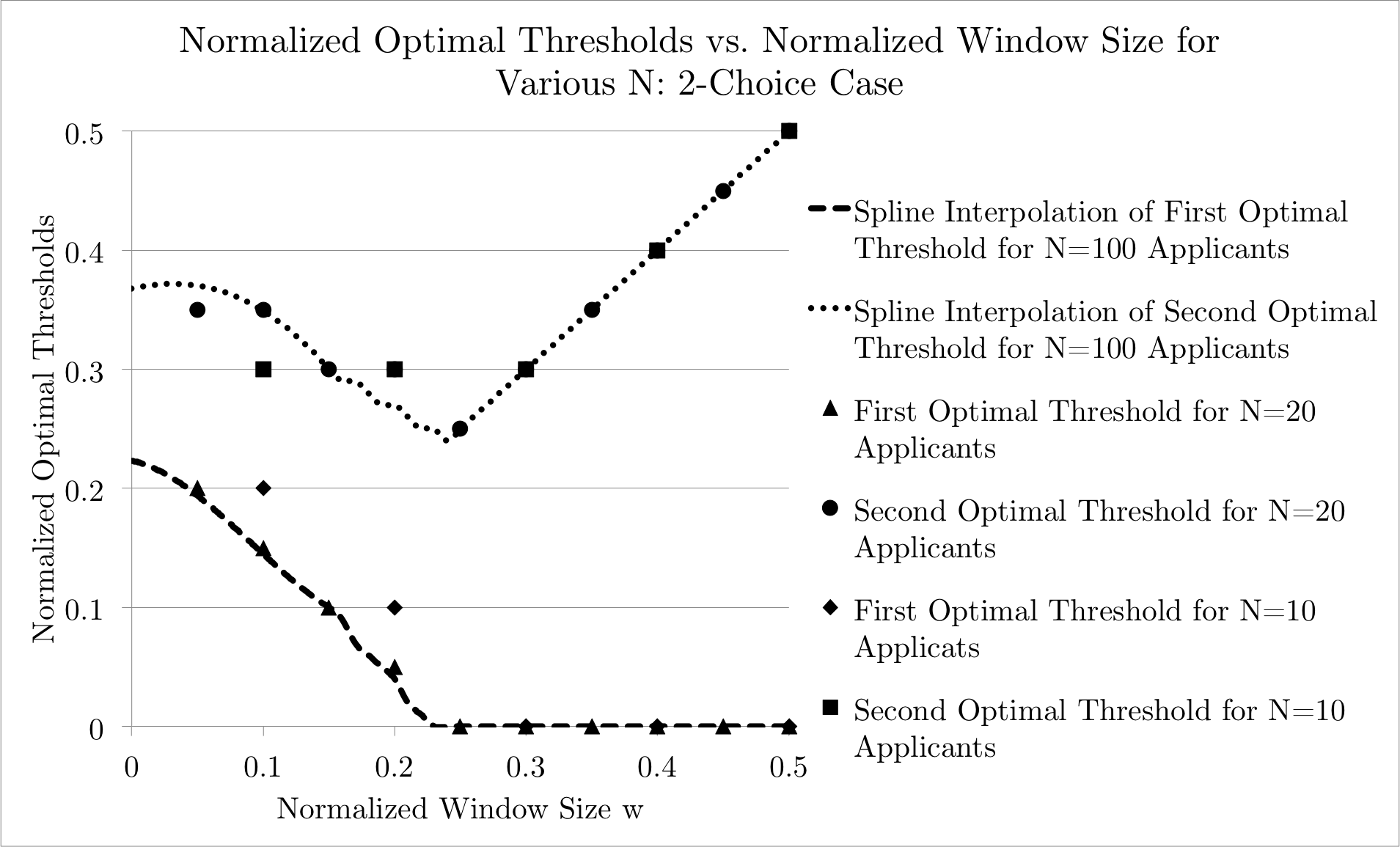}
\caption{The Variation of the Normalized Thresholds $\rho_1^{*} =\frac{\delta^*_1}{N}$ and $\rho_2^*=\frac{\delta^*_2}{N}$ with the Normalized Window Size $w=\frac{K}{N}$ for $N=10$, $20$, $100$}
\label{Fig2CT}
\end{figure}

Therefore $\Pr(\Win) \approx \tau_1(1)+\tau_2(1)+\tau_3(1)$, and as a result, $\Pr(\Win)$ only depends on $w$, $\rho_1$, and $\rho_2$ for large $N$. As with previous cases, the optimal $\Pr(\Win)$, the normalized first optimal threshold $\rho^*_1$, and the normalized second optimal threshold $\rho^*_2$ only depend on the normalized window size $w$. Figure \ref{Fig2CT} shows how the normalized thresholds depend on the normalized window size $w=\frac{K}{N}$. Similarly, Figure \ref{Fig2CP} shows how the optimal probability of winning depends on the normalized window size $w=\frac{K}{N}$. 

\begin{figure} [h]
\centering
\includegraphics[width=\textwidth]{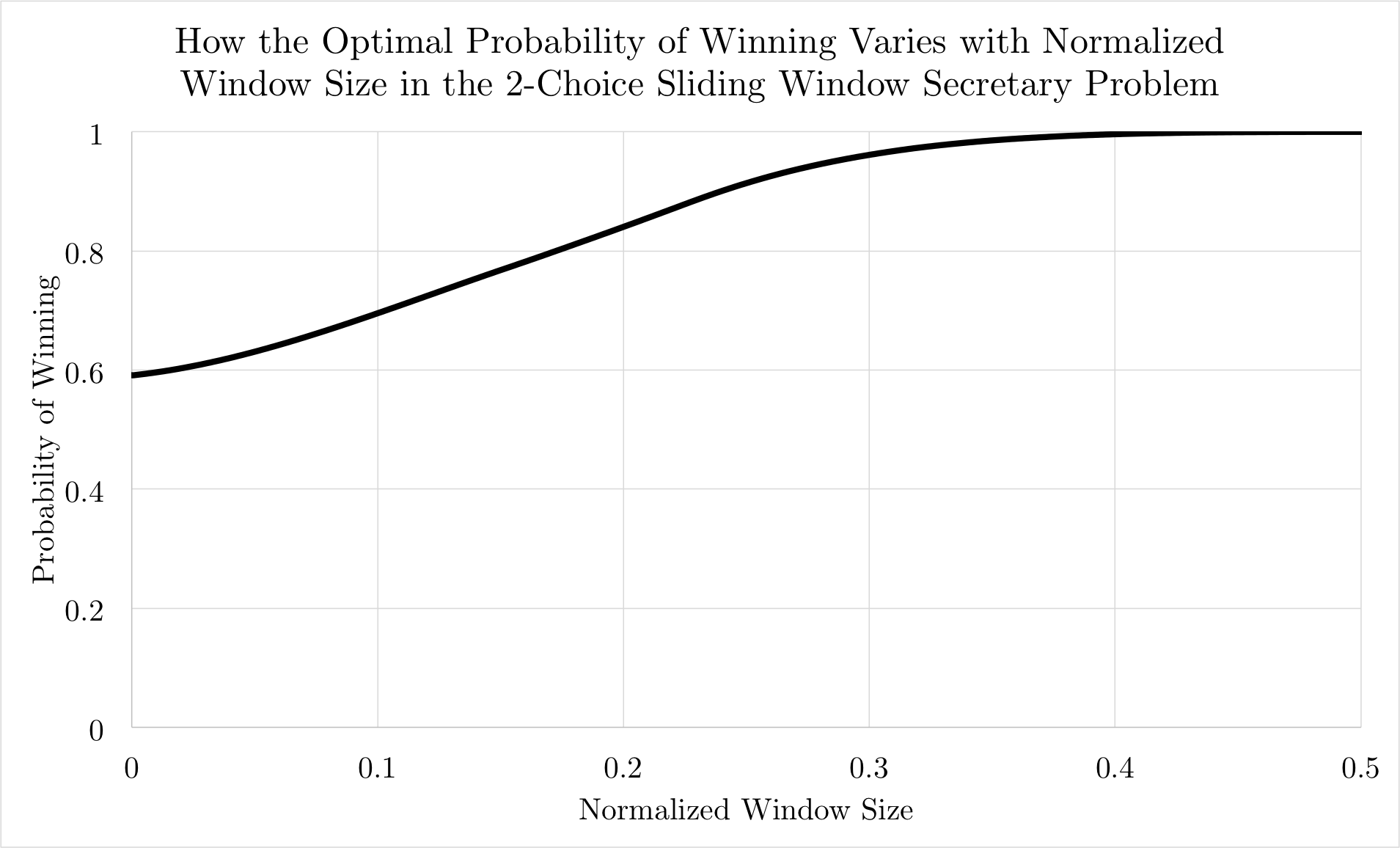}
\caption{The Variation of the 2-Choice Optimal Probability of Winning with the Normalized Window Size $w=\frac{K}{N}$. For $K\geq \frac{N}{2}$, $\Pr(\Win)=1$, and thus the graph only displays probabilities of winning for $K<\frac{N}{2}$.}
\label{Fig2CP}
\end{figure}

\section{Conclusions and Directions for Future Research} \label{concluding}
We studied the Sliding-Window Secretary Problem for 3 different cases for a fixed number
of applicants: (i) choosing the best, (ii) choosing either the best or second best, and 
(iii) two choices to choose the best.
For each case, we found the maximum probability of winning for any window size, computed the optimal thresholds, and performed asymptotic analysis.

Our results naturally extend to the Top-$L$ case, where the interviewer wins
if one of the top $L$ is chosen. For future research directions, the Sliding-Window 
can also apply to the other classical Secretary problems. For example, finding the
best expected rank, in which Chow \textit{et al.} \cite {Ref3} have already found the best expected rank for a sliding window of size 1, while Goldys \cite {Ref5} has found the best expected rank for a sliding window of size 2.
Another extension is the full-information problem where the interviewer knows a cardinal score of each applicant and the probability distribution of scores, instead of just relative ranks. This variation is more applicable to realistic situations, because decisions are made not solely based on the ordinal ranks of options but the magnitude of the benefit of each option.

\section{Acknowledgments}
We thank Professor Daniel Kleitman for the original inspiration of the problem.
We would also like to thank Dr.\ Tanya Khovanova for proposing this variation of the problem. We would also like to thank Dr.\ Scott Kominers, Shashwat Kishore, Dr.\ John Rickert, Siddhartha Jena, Richard Yip, and Dr.\ Andrew Charman for providing feedback on the paper. We finally would like to thank the Research Science Institute and the Center for Excellence in Education for giving us the opportunity to perform this research in the Department of Mathematics at the Massachusetts Institute of Technology.

\newpage
\begin{singlespace}
\bibliographystyle{plain}
\bibliography{main.bib}

\begin{thebibliography}{10}

\bibitem{Ref10}
Miklos Atjai, Nimrod Megiddo, and Orli Waarts.
\newblock Improved algorithms and analysis for secretary problems and
  generalizations.
\newblock {\em Journal of Discrete Math}, 14(1):1--27, 1995.

\bibitem{Ref9}
Luca Becchetti and Elias Koutsoupias.
\newblock Competitive analysis of aggregate max in windowed streaming.
\newblock In Susanne Albers, Alberto Marchetti-Spaccamela, Yossi Matias,
  Sotiris Nikoletseas, and Wolfgang Thomas, editors, {\em Automata, Languages
  and Programming}, pages 156--170. Springer Berlin Heidelberg, 2009.

\bibitem{Ref3}
Y.~S. Chow, S.~Moriguti, H.~Robbins, and S.~M. Samuels.
\newblock Optimum selection based on relative rank (the `secretary problem').
\newblock {\em Israel Journal of Mathematics}, 2:81--90, 1964.

\bibitem{Ref16}
Jnaneshwar Das.
\newblock {\em Data-driven Robotic Sampling for Marine Ecosystem Monitoring}.
\newblock PhD thesis, University of Southern California, 2014.

\bibitem{Ref7}
E.~B. Dynkin.
\newblock The optimal choice of the stopping moment for a markov process.
\newblock {\em Dokl. Akad. Nauk. SSSR}, 150:238--240, 1963.

\bibitem{Ref1}
Thomas~S. Ferguson.
\newblock Who solved the secretary problem?
\newblock {\em Statistical Science}, 4:282--296, 1989.

\bibitem{Ref2}
John~P. Gilbert and Frederick Mosteller.
\newblock Recognizing the maximum of a sequence.
\newblock {\em Journal of the American Statistical Association}, 61:35--73,
  1966.

\bibitem{Ref12}
Yogesh Girdhar and Gregory Dudek.
\newblock Optimal online data sampling or how to hire the best secretaries.
\newblock {\em Computer and Robot Vision}, 14:292--298, 2009.

\bibitem{Ref5}
B.~Goldys.
\newblock The secretary problem - the case with memory for one step.
\newblock {\em Demonstratio Mathematica}, 11:789--799, 1978.

\bibitem{Ref6}
D.~V. Lindley.
\newblock Dynamic programming and decision theory.
\newblock {\em Applied Statistics}, 10:39--52, 1961.

\bibitem{Ref8}
Joseph~D. Petrucelli.
\newblock Best-choice problems involving uncertainty of selection and recall of
  observations.
\newblock {\em Journal of Applied Probability}, 18:415--425, 1981.

\bibitem{Ref15}
Darryl~A. Seale and Amnon Rapoport.
\newblock Sequential decision making with relative ranks: An experimental
  investigation of the “secretary problem”.
\newblock {\em Organizational Behavior and Human Decision Processes},
  69(3):221--236, 1997.

\end{thebibliography}
\end{singlespace}

\newpage
%
\appendix
\gdef\thesection{Appendix \Alph{section}}
\renewcommand{\thetheorem}{\Alph{section}\arabic{theorem}}
\renewcommand{\thecorollary}{\Alph{section}\arabic{corollary}}
\renewcommand{\thelemma}{\Alph{section}\arabic{lemma}}

\newpage
 
\gdef\thesection{Appendix \Alph{section}}
\renewcommand{\thetheorem}{\Alph{section}\arabic{theorem}}
\renewcommand{\thecorollary}{\Alph{section}\arabic{corollary}}
\renewcommand{\thelemma}{\Alph{section}\arabic{lemma}}

\section{The Classical Secretary Problem}  \label{Ap0}

We present Gilbert and Mosteller's proofs of the optimal strategy for the secretary problem.

\begin {theorem} [Gilbert and Mosteller (1966) \cite{Ref2}] \label{ThmImpt}
The solution can be restricted to the strategy in which for some integer $d^{*} \geq 1$, the interviewer rejects the first $d^{*}$ applicants and chooses the next applicant who is better than the first $d^{*}$ applicants \cite{Ref2}.

\end{theorem}

We present the proof of Theorem \ref{ThmImpt} to demonstrate the methods used to prove the optimal form of the strategy for the Sliding Window Secretary Problems.

\begin{proof} \label{hello}
We shall define a \textit{candidate} as an applicant such that if the interviewer chooses to accept this applicant, the probability of winning is strictly nonzero. 

Suppose we reach a candidate with index $i$. Because candidate $i$ is better than all previous applicants, the probability that candidate $i$ is the best applicant is $\frac{i}{N}$.

Now consider the probability of winning with the optimal strategy given that the interviewer rejects candidate $i$. The probability of winning given that candidate $i$ is rejected decreases as $i$ increases, because the larger $i$ is, the more likely that the best applicant is at $i$. 

The interviewer only chooses applicant $i$ if $i$ is a candidate and if the following inequality holds.
\begin{equation} \label{Eq0}
\Pr(\Win \mid \mbox{Choosing Candidate } i) > P(\Win \mid \mbox{Skipping Candidate } i)
\end{equation} 

Because there is a $\frac{1}{N}$ chance that applicant $N$ is better than applicant $N-1$ if applicant $N-1$ is a candidate, $\Pr(\Win \mid \mbox{Skipping Candidate } N-1) =\frac{1}{N}$. Therefore, $\Pr(\Win \mid \mbox{Choosing Candidate } N-1) > \Pr(\Win \mid \mbox{Skipping Candidate } N-1)$. Because the probability of winning given that candidate $i$ is chosen strictly increases in $i$, the probability of winning given that candidate $i$ is rejected decreases in $i$, there is a greatest integer $d^{*} \in [0,N-1]$ such that Inequality \ref{Eq0} only holds after $d^{*}$.
\end{proof}

We now find the optimal threshold $d^{*}$ for large $N$, as seen in Gilbert and Mosteller (1966)

\begin {theorem} [Gilbert and Mosteller (1966) \cite{Ref2}]
For $N$ large, $d^{*}\approx \frac{N}{e}$.
\end{theorem}

\begin{proof}
Let $d$ be an arbitrary threshold value, and $d^{*}$ be the value of $d$ that maximizes the probability of winning. Let $\Pr(\Win \mid j)$ be the probability of finding the best-ranked applicant if $j$ is the index of the best-ranked applicant. Then $\Pr(\Win \mid j)=0$ for $j\leq d$, and because the probability of finding the best-ranked applicant after $d$ is equal to the probability of not finding an applicant better than the first $d$ applicants in the first $j-1$ applicants, $\Pr(\Win \mid j)=\frac{d}{j-1}$ for $j>d$. 

Each value of $j$ occurs with probability $\frac{1}{N}$, and thus by the total probability theorem, $$\Pr(\Win)=\sum\limits_{j=1}^N \frac{\Pr(\Win \mid j)}{N}= \frac{1}{N} \sum\limits_{j=d+1}^N \left(\frac{d}{j-1}\right)$$.

$\Pr(\Win)$ can be converted to an integral for $N$ large:
$$\Pr(\Win) \approx \frac{d}{N}  \int_{\frac{d}{N}}^{1} \frac{1}{t} dt.$$

So $\Pr(\Win) \approx -\frac{d}{N} \log (\frac{d}{N})$. The value of $d$ that maximizes $\Pr(\Win)$ is $\frac{N}{e}$. So $d^{*} = \frac{N}{e}$, and $\Pr(\Win) = \frac{1}{e}$ \cite{Ref2}.
\end{proof}

\newpage

\newpage
\gdef\thesection{Appendix \Alph{section}} 

\gdef\thesection{Appendix \Alph{section}} 

\section{Optimal Threshold Values for Various Numbers of Applicants and Window Sizes: Choosing the Best Applicant}  \label{Ap1}

\begin{table} [h]
\begin{center}

\begin{tabular}{|l|l|l|l|}\hline
Number of  & Window & Optimal & Probability \\ 
Applicants & Size & Threshold & of Win  \\ \hline \hline
$6$ & $2$ & $1$ & $0.5611$ \\ \hline
$6$ & $3$ & $0$ or $1$ & $0.7167$\\ \hline \hline
$7$ & $2$ & $2$ & $0.5321$\\ \hline
$7$ & $3$ & $1$ & $0.6690$ \\ \hline
$7$ & $4$ & $0$ & $0.8114$\\ \hline \hline
$8$ & $2$ & $2$ & $0.5089$\\ \hline
$8$ & $3$ & $1$ & $0.6199$ \\ \hline
$8$ & $4$ & $0$ or $1$ & $0.7405$\\ \hline
$8$ & $5$ & $0$ & $0.8655$\\ \hline \hline
$9$ & $2$ & $3$ & $0.4880$ \\ \hline
$9$ & $3$ & $2$ & $0.5741$\\ \hline
$9$ & $4$ & $1$ & $0.6988$\\ \hline
$9$ & $5$ & $0$ & $0.8099$ \\ \hline
$9$ & $6$ & $0$ & $0.8988$\\ \hline \hline
$10$ & $2$ & $3$ & $0.4774$ \\ \hline
$10$ & $3$ & $2$ & $0.5634$\\ \hline
$10$ & $4$ & $1$ & $0.6566$\\ \hline
$10$ & $5$ & $0$ or $1$ & $0.7544$ \\ \hline
$10$ & $6$ & $0$ & $0.8544$\\ \hline
$10$ & $7$ & $0$ & $0.9210$\\ \hline
\end{tabular}

\end{center}
\caption{Optimal Values for The Threshold Value $d^{*}$ Given the Number of Applicants $N$ and Window Size $K$, along with Respective Probabilities of Success $\Pr(\Win)$}
\label{somenumbers}
\end{table}

\newpage

\gdef\thesection{Appendix \Alph{section}} 

\section{Probabilities of Success Given a Threshold $d$ for $100$ Applicants and a Window Size of $2$: Choosing the Best Applicant}  \label{Ap2}
\begin{table} [h]
\begin{center}
\begin{tabular} {|l|l|} \hline
Value of & Probability \\ 
Threshold $d$ & of Win \\ \hline
$33$ & $0.3760$ \\ \hline
$34$ & $0.3768$ \\ \hline
$35$ & $0.3773$ \\ \hline
$36$ & $0.3775$ \\ \hline
$37$ & $0.3774$ \\ \hline
\end{tabular}
\end{center}
\caption{Probabilities of Success For Certain Thresholds Given $N=100$ and $K=2$. The optimal $d^{*}$ for $K=2$ is $36$, which is the optimal $d^{*}$ for the original secretary problem.}
\label{k2}
\end{table}

\newpage

\gdef\thesection{Appendix \Alph{section}} 

\section{Values of Optimal Thresholds ($d^{*}$) for Different Window Sizes ($k$) and Large Number of Applicants ($n$)}  \label{Ap3}

\begin{table} [h]
\begin{center}
\begin{tabular} {|l|l|} \hline
Normalized & Normalized Value  \\ 
Window Size & of Optimal Threshold  \\ \hline
$0$	& $0.3679$ \\ \hline
$0.2$ & $0.2635$ \\ \hline
$0.22$&$0.2494$ \\ \hline
$0.24$ &	$0.2347$ \\ \hline
$0.26$&	$0.2193$ \\ \hline
$0.28$ &$0.2033$ \\ \hline
$0.3$ &	$0.1867$ \\ \hline
$0.32$ &	$0.1696$ \\ \hline
$0.34$&$0.1520$ \\ \hline
$0.36$ & $0.1341$ \\ \hline
$0.38$ & $0.1158$ \\ \hline
$0.4$ & $0.09716$ \\ \hline
$0.42$	& $0.07823$ \\ \hline
$0.44$ &	$0.05903$ \\ \hline
$0.46$	& $0.03958$ \\ \hline
$0.48$	& $0.0199$ \\ \hline
$0.5$ &$0$ \\ \hline
\end{tabular}
\end{center}
\caption{Normalized Threshold Values $\frac{d^{*}}{N}$ for Select Values of the Normalized Window Size $\frac{K}{N}$}
\label{k3}
\end{table}

\newpage
\gdef\thesection{Appendix \Alph{section}} 

\section{Optimal Threshold Values for Various Numbers of Applicants and Window Sizes: Choosing One of the Best Two Applicants}  \label{Ap4}

\begin{table} [h]
\begin{center}
\begin{tabular}{|l|l|l|l|l|}\hline
Number of & Window & Optimal  & Optimal  & Probability\\
Applicants & Size & Threshold 1 & Threshold 2 & of Winning \\ \hline \hline
$4$ & $2$ & $1$ & $1$ & $0.9167$ \\ \hline \hline
$5$ & $2$ & $1$ & $2$ & $0.8833$\\ \hline 
$5$ & $3$ & $1$ & $1$ & $0.9667$\\ \hline \hline
$6$ & $2$ & $1$ & $3$ & $0.8333$ \\ \hline
$6$ & $3$ & $1$ & $2$ & $0.9333$\\ \hline 
$6$ & $4$ & $1$ & $1$ & $0.9833$\\ \hline \hline
$7$ & $2$ & $2$ & $4$ & $0.7929$ \\ \hline
$7$ & $3$ & $1$ & $3$ & $0.8976$\\ \hline
$7$ & $4$ & $1$ & $2$ & $0.9571$\\ \hline 
$7$ & $5$ & $1$ & $1$ & $0.9905$ \\ \hline \hline
$8$ & $2$ & $2$ & $4$ & $0.7696$\\ \hline
$8$ & $3$ & $1$ & $3$ or $4$ & $0.8595$\\ \hline
$8$ & $4$ & $1$ & $3$ & $0.9310$ \\ \hline
$8$ & $5$ & $1$ & $2$ & $0.9702$\\ \hline \hline
$8$ & $6$ & $1$ & $1$ & $0.9940$ \\ \hline
$9$ & $2$ & $2$ & $5$ & $0.7454$\\ \hline
$9$ & $3$ & $2$ & $4$ & $0.8364$\\ \hline
$9$ & $4$ & $1$ & $3$ & $0.9052$\\ \hline
$9$ & $5$ & $1$ & $2$ & $0.9517$ \\ \hline
$9$ & $6$ & $1$ & $1$ & $0.9788$\\ \hline
$9$ & $7$ & $1$ & $1$ & $0.9960$\\ \hline
\end{tabular}
\end{center}
\caption{Optimal Values for The Thresholds Given the Number of Applicants and Window Size, along with Respective Probabilities of Success}
\label{somenumbers}
\end{table}

\end{document}